\documentclass[12pt]{amsart}

\usepackage{amsmath,amssymb}
\usepackage[margin=2.5cm]{geometry} 
\usepackage{color}
\usepackage{tikz}
\usepackage{hyperref}
\usepackage{bbold}

\newcommand{\ca}{\mathcal{C}}

\newcommand{\da}{\mathcal{D}}
\newcommand{\Z}{\mathbb{Z}}
\newcommand{\SL }{{\rm SL}}
\newcommand{\C}{\mathbb{C}}
\newcommand{\N}{\mathbb{N}}
\newcommand{\GCD}{{\rm GCD}}
\newcommand{\End}{{\rm End}}
\newcommand{\Spec}{{\rm Spec}}
\newcommand{\Irr}{{\rm Irr}}
\newcommand{\Hom}{{\rm Hom}}

\def\mymathhyphen{{\hbox{-}}}

\newcommand{\TODO}[1]
{}

\theoremstyle{plain} 
\newtheorem{theo}{Theorem}[section]

\newtheorem{cor}[theo]{Corollary}
\newtheorem{prop}[theo]{Proposition}

\theoremstyle{definition} 
\newtheorem{defi}[theo]{Definition}
\newtheorem{rem}[theo]{Remark}
\newtheorem*{rem*}{Remark}

\begin{document}


\begin{abstract}
We give formulae for the multiplicities of eigenvalues of generalized rotation operators in terms of generalized Frobenius-Schur indicators in a semisimple spherical tensor category $\mathcal{C}$.  In particular, this implies that the entire collection of rotation eigenvalues for a fusion category can be computed from the fusion rules and the traces of rotation at finitely many tensor powers.  We also establish a rigidity property for FS indicators of fusion categories with a given fusion ring via Jones's theory of planar algebras. If $\mathcal{C}$ is also braided, these formulae yield the multiplicities of eigenvalues for a large class of braids in the associated braid group representations.  When $\mathcal{C}$ is modular, this allows one to determine the eigenvalues  and multiplicities of braids in terms of just the $S$ and $T$ matrices. 
\end{abstract}

\title{Eigenvalues of rotations and braids in spherical fusion categories}
\author{Daniel Barter, Corey Jones and Henry Tucker}
\date{\today}
\maketitle

\section{Introduction} \label{sec:Introduction}

Our understanding of symmetry in a diverse range of topology, representation theory, and mathematical physics relies on first understanding {\it rotations}. Obvious examples appear in the study of winding numbers and knot polynomials. The actions of rotation operators played an important role in the construction and obstruction theory of Vaughan Jones's {\it planar algebras} \cite{JonPA}, which provide a diagrammatic axiomatization of the representation theory of Murray-von Neumann subfactors.  In particular, the small index subfactor classification program has been advanced significantly by understanding rotation eigenvalues in relation to other data; for an overview, see \cite{Jonann}, \cite{ind50}. 

From a purely algebraic standpoint rotation can be realized as a tensor flip map:
\[
v_1 \otimes v_2 \otimes \cdots \otimes v_n \mapsto v_n \otimes v_1 \otimes v_2 \otimes \cdots \otimes v_{n-1}
\]
Kashina, Sommerh\"auser, and Zhu \cite{KYS} showed that the traces of rotation operators on tensor powers of a given representation of a semisimple Hopf algebra yield the {\it higher Frobenius-Schur (FS) indicators} for the representation. (The definition of these in the Hopf algebra settings generalizes the classical Frobenius-Schur indicators defined for complex representations of finite groups; see \cite{lm}.)

Ng and Schauenberg used a synthesis of the subfactor and Hopf algebra approaches to generalize much of this theory to the setting of {\it spherical fusion categories} \cite{NS1}\cite{NS3}. The advantages of this setting are twofold. First, generality is sufficient to unify the representation theories for quantum groups \cite{ResTur91}, subfactors \cite{Mu1}, conformal nets, and vertex operator algebras \cite{KLM} along with the classification of fully extended 3-dimensional topological field theories \cite{DSS}. Second, restrictions are sufficient to allow extensive use of graphical methods for morphisms.

The categorical FS indicators have proven to be a powerful invariant of fusion categories as well as their braided counterparts. They provide the main technical ingredient in the recent proofs showing the representation of the modular group $SL_{2}(\Z)$ coming from the Drinfel'd center of a spherical fusion category (that is, the canonically associated braided fusion category) has a congruence subgroup as its kernel \cite{SZ}\cite{NS2} and showing rank finiteness for {\it modular} categories \cite{BNRW}. Crucially, Ng and Schauenburg provide a method for computing FS indicators of a spherical fusion category in terms of the modular data of its Drinfel'd center \cite{NS3}. The preceding is discussed in greater detail in section 2.

We have seen that there is a great deal of theory regarding the \textit{traces} of rotation operators and how to compute them from basic categorical data, however their actual eigenvalues (and corresponding multiplicities) have been relatively inaccessible except in special cases.  In this note we will give formulae for the multiplicities of the eigenvalues of (generalized) rotation operators in terms of the (generalized) FS indicators, and we will present two applications. 

In section 3 we establish the eigenvalue multiplicities using Galois actions on cyclotomic fields and properties of (generalized) FS indicators to derive the traces of the powers of rotations.  The finite Fourier transform yields the desired formulae in Theorem \ref{thm:main_theorem}.  However, it is desirable to be able to compute the rotation eigenvalues from as little information as possible since there are many situations where the full modular data for the center is not easily accessible. For example, in \cite{GM}, Gannon and Morrison describe a procedure for computing possible modular data for the Drinfel'd centers $Z(\ca)$ of categorifications $\ca$ of a given fusion ring. Their algorithm effectively produces finitely many possibilities for the forgetful functor $Z(\ca)\to \ca$ and the $T$-matrix of $Z(\ca)$, but finding the $S$-matrix requires solving a generally large system of quadratics which may or may not be possible. In Theorem \ref{moreforless} we show how to compute the (generalized) Frobenius-Schur indicators from just the $T$-matrix and the forgetful functor, hence our formulae let us determine the set of rotation eigenvalues for a given categorification from this information alone.

The FS indicator data has been an important invariant for classifying {\it categorifications} $\ca$ of a given fusion ring $R$. This is a fusion category whose tensor product structure is isomorphic to $R$ on the level of rings. At the end of section 3 we establish a rigidity property for the FS indicators of categorifications of a given fusion ring. Rigidity phenomenon describe situations where one gets ``more for less'':  some {\it a priori} weaker invariants being equal implies some stronger invariants are equal.  FS indicators are invariants of spherical fusion categories; we say that a fusion ring has {\it FS indicator rigidity} if each of its categorifications has a unique sequence of FS indicators. (If $\ca$ is a fusion category, the sequence is periodic, so this is a finite set of data.)  For example, the Tambara-Yamagami fusion rings have FS indicator rigidity \cite{BJ}.  Richard Ng has asked whether {\it singly-generated} fusion rings have FS indicator rigidity.  We observe a negative answer to this question: the quadratic  Haagerup fusion ring with $G=\mathbb{Z}_{3}$ admits two inequivalent, unitary (hence spherical) categorifications with the same FS indicator data.  

However, we are able to show a weaker rigidity property: if two categorifications of a given singly-generated fusion ring have the same higher Frobenius-Schur indicator data then their associated {\it planar algebras} are isomorphic {\it as objects} in the representation category of Jones's annular Temperley-Lieb algebroid \cite{Jonann}. As we shall see this category is related closely to the Drinfel'd center of the category of representations of the Drinfel'd-Jimbo quantum group for $\mathfrak{sl}_2$. We establish this property of fusion rings in Theorem \ref{annFSrigid} and name it {\it annular FS indicator rigidity}. A planar algebra is a {\it commutative algebra object} in the annular representation category.  Given an object $v\in \ca$ one can construct a planar algebra $P^{v}$. If $v\otimes \overline{v}$ tensor generates $\ca$ one can recover the category $\ca$ from the planar algebra. Suppose $R$ is a fusion ring, and $\mathcal{C}$ and $\mathcal{D}$ are spherical categorifications of $R$.  Suppose $v\in \ca$ and $w\in \mathcal{D}$ are identified under the equivalence of fusion rings. Then the annular FS rigidity property is stated as follows: if the FS indicator data for $\ca$ and $\mathcal{D}$ are the same then $P^{v}\cong P^{w}$ as objects but {\it not necessarily} as commutative algebras. 

One can only recover the category $\ca$ from $P^v$ as an {\it algebra} object, hence the failure of full FS indicator rigidity can be expressed as {\it non-uniqueness} of the commutative algebra structure on a given object in the annular representation category. This also suggests a possible approach to determining which classes of fusion rings exhibit the stronger property of {\it full} FS indicator rigidity.  For example, it is possible that from general considerations one could determine that certain objects in the annular representation category admit \textit{at most one} commutative algebra structure, which would imply FS indicator rigidity for any underlying fusion ring.

Finally, in section 4 we use our formulae to compute eigenvalues of braid group representations associated to {\it braided} spherical (= {\it ribbon}) fusion categories in terms of (generalized) FS indicator data. In particular, when the braiding is non-degenerate (i.e. the category is {\it modular}) our formulae can be written in terms of the modular data, i.e. the $S$ and $T$ matrices from the braiding and spherical structures. 

Given an object $v$ in a ribbon fusion category $\ca$, we have representations of the braid groups $B_n$:
\[
\pi_{n, v}: B_{n} \rightarrow {\rm End}_\ca(v^{\otimes n})
\]
Using our formulae in Proposition \ref{thm:main_theorem} and Proposition \ref{moreforless} we can compute the eigenvalues under these representations for a large class of braids , called the {\it Jucys-Murphy elements} of the braid group $B_n$:
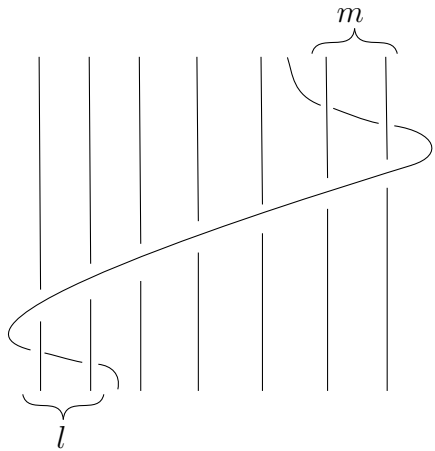
\[
\begin{tikzpicture}[yscale=-1,scale=0.0333,baseline={([yshift=-.5ex]current bounding box.center)}]
\begin{scope}[shift={(0.00mm,719.29mm)}]
\draw [fill=none,draw=black] (157.14mm,-488.75mm)
-- ++(5.72mm,930.22mm)
;
\draw [fill=none,draw=black] (162.85mm,568.51mm)
-- ++(0.00mm,277.77mm)
;
\draw [fill=none,draw=black] (357.14mm,-488.75mm)
-- ++(5.72mm,826.86mm)
;
\draw [fill=none,draw=black] (362.86mm,480.23mm)
-- ++(0.00mm,366.06mm)
;
\draw [fill=none,draw=black] (557.14mm,-488.75mm)
-- ++(5.72mm,747.18mm)
;
\draw [fill=none,draw=black] (562.86mm,407.02mm)
-- ++(0.00mm,439.27mm)
;
\draw [fill=none,draw=black] (787.14mm,-488.75mm)
-- ++(5.73mm,656.74mm)
;
\draw [fill=none,draw=black] (792.86mm,293.33mm)
-- ++(0.00mm,552.96mm)
;
\draw [fill=none,draw=black] (1043.14mm,-488.74mm)
-- ++(5.73mm,589.99mm)
;
\draw [fill=none,draw=black] (1048.86mm,216.33mm)
-- ++(0.00mm,629.96mm)
;
\draw [fill=none,draw=black] (1303.14mm,-488.74mm)
-- ++(5.73mm,484.48mm)
;
\draw [fill=none,draw=black] (1308.86mm,118.76mm)
-- ++(0.00mm,727.53mm)
;
\draw [fill=none,draw=black] (1541.14mm,-488.74mm)
-- ++(5.73mm,410.68mm)
;
\draw [fill=none,draw=black] (1546.86mm,34.56mm)
-- ++(0.00mm,811.73mm)
;
\draw [fill=none,draw=black] (471.43mm,839.83mm)
.. controls (483.55mm,779.55mm) and (446.57mm,745.83mm) .. (393.57mm,739.70mm)
;
\draw [fill=none,draw=black] (329.28mm,732.16mm)
.. controls (278.42mm,722.78mm) and (203.99mm,699.52mm) .. (178.57mm,694.48mm)
;
\draw [fill=none,draw=black] (122.85mm,682.64mm)
.. controls (-394.46mm,551.75mm) and (1530.12mm,-18.55mm) .. (1645.72mm,-55.94mm)
;
\draw [fill=none,draw=black] (1645.30mm,-56.11mm)
.. controls ++(131.54mm,-40.41mm) and ++(165.50mm,25.43mm) .. ++(-69.98mm,-157.06mm)
;
\draw [fill=none,draw=black] (1512.86mm,-222.82mm)
.. controls ++(-63.20mm,-12.67mm) and ++(53.35mm,19.52mm) .. ++(-181.43mm,-60.29mm)
;
\draw [fill=none,draw=black] (1281.43mm,-294.95mm)
.. controls ++(-111.91mm,-42.90mm) and ++(22.01mm,70.81mm) .. ++(-132.86mm,-191.64mm)
;
\draw [fill=none,draw=black] (1246.38mm,-501.23mm)
.. controls ++(19.59mm,-88.32mm) and ++(-9.57mm,103.14mm) .. ++(170.51mm,-101.06mm)
;
\draw [fill=none,draw=black] (1587.16mm,-501.23mm)
.. controls ++(-19.59mm,-88.32mm) and ++(9.57mm,103.14mm) .. ++(-170.51mm,-101.06mm)
;
\draw [fill=none,draw=black] (91.79mm,859.80mm)
.. controls (110.40mm,948.13mm) and (244.65mm,857.72mm) .. (253.74mm,960.87mm)
;
\draw [fill=none,draw=black] (415.46mm,859.80mm)
.. controls (396.86mm,948.13mm) and (262.60mm,857.72mm) .. (253.51mm,960.87mm)
;
\node [black] at (1400.14mm,-653.35mm) { $m$ };
\node [black] at (245.71mm,1029.51mm) { $l$ };
\end{scope}
\end{tikzpicture}
\]


This class of braids includes the braid group generators and the periodic braids.  As demonstrated in \cite{TW,RT}, the eigenvalues of the braid generator may contain a large amount of information about the associated representation in low dimensions.   Indeed, small representations of $B_3$ are essentially determined by the spectrum of the braid group generator \cite{TW}. In many cases related to quantum groups, the spectrum for the braiding can be worked out using explicit forms of the braid matrix, but, until now, there has been no general procedure for finding which square roots occur and with what multiplicity.  We also anticipate the utility of our formulae for understanding Drinfel'd centers of ``exotic'' fusion categories which don't appear to come from quantum groups in an obvious way; see \cite{HRW}.

\subsection*{Remark} Our results only require $\ca$ to be a spherical semisimple tensor category, that is with possibly infinitely many isomorphism classes of simple objects. We primarily restrict our attention to fusion categories (finitely many simple objects) since in this case the sequence of higher FS indicators is periodic, but we have been careful to prove our results in a way that generalizes easily to the non-fusion case.


\subsection*{Acknowledgments}   The authors wish to thank Vaughan Jones for the initial question of whether rotation eigenvalues can be obtained from FS indicators. The second and third author wish to thank the Hausdorff Research Institute for Mathematics (HIM) for their kind hospitality during the beginning stages of this project. We also thank the organizers of the HIM Von Neumann Algebras Trimester Program for their invitations, with a special thanks to contact organizer Dietmar Bisch. All three authors also thank Scott Morrison, Susan Montgomery, and Richard Ng for many useful conversations. We also extend thanks to the referee for many helpful suggestions. The second author was supported by Discovery Projects ``Subfactors and symmetries'' DP140100732 and ``Low dimensional categories'' DP160103479 from the Australian Research Council.

\section{Preliminaries} \label{sec:Preliminaries}

\subsection{Fusion categories} In this note {\bf tensor category} means a locally finite $k$-linear abelian monoidal category $(\ca, \otimes, \mathbf{1}, \alpha, \lambda, \rho)$ where the monoidal structure $\otimes$ is bilinear on morphisms and the unit object has $\End_\ca(\mathbf{1})\cong k$. The local finiteness condition requires objects to have finite length and morphism spaces to be finite dimensional. Refer to \cite{ENOG} for these definitions and other background on the theory of tensor categories. 

We assume, for convenience, that our tensor categories are strict, which can be done without loss of generality by Mac Lane's Strictness Theorem. With this assumption we may take advantage of {\it graphical calculus} for morphisms. 

\begin{rem*} Our graphical calculus convention is that {\it diagrams are read bottom to top}, the opposite convention to \cite{NS1,NS2}, but the same as \cite{Mu1,Mu2}.
\end{rem*}

\begin{defi}
A {\bf fusion category} is a {\it finite} tensor category $(\ca, \otimes, \mathbf{1}, \alpha, \lambda, \rho)$ that is also:
\begin{itemize}
    \item {\it semi-simple} with set of simple objects $\Irr(\ca)$ such that $\mathbf{1}\in \Irr(\ca)$, and
    \item {\it rigid}, i.e. objects $x\in\ca$ have (right \& left) {\it duals} $\overline{x}\in \ca$
\end{itemize}
({\it Note:} the finiteness condition implies that $|\Irr(\ca)|<\infty$)
\end{defi}

The first important invariant of fusion categories is the {\bf fusion ring}, which is the same as the Grothendieck ring $K_0(\ca)$. This is the $\mathbb{Z}$-based ring with basis given by the set $\text{Irr}(\ca)$ of isomorphism classes of simple objects in $\ca$, addition by the direct sum, and multiplication by the direct product. If $a,b,c \in \text{Irr}(\ca)$ then the structure coefficients are the integers $N^{c}_{a,b}:= \dim \Hom_\ca (a\otimes b, c)$. In general, we use the notation $N^{b}_{a_1, \dots, a_n}:= \dim \Hom_\ca(a_1\otimes \dots a_{n}, b)$.  These larger multiplicities can be computed from the fusion rule in the following way: 
\[N^{d}_{a,b,c}=\sum_{e\in \text{Irr}(\ca)} N^{d}_{a,e}N^{e}_{b,c}.\]
We also use the notation $a^{n}:=a\otimes \dots \otimes a$ so that, for example, $N^{d}_{a,a,a}=N^{d}_{a^{3}}$. These numbers depend on the objects only up to isomorphism. Given a fusion ring $R$ we call a fusion category $\ca$ a {\bf categorification} of $R$ if $K_0(\ca)\cong R$.

\subsection{Pivotal structures}
In this article we consider {\it spherical} fusion categories; these come with a well-defined trace on endomorphisms of objects. A {\bf pivotal structure} on a rigid tensor category is a choice of natural monoidal isomorphism from the double dual endofunctor to the identity endofunctor. It was shown in \cite{NS3} that we can always choose our pivotal structure to be \textit{strictly pivotal} in the sense that we can pick our duality functor so that the natural isomorphism from the double dual functor to the identity functor is the identity. Such a choice will be more conducive to graphical methods; see the following discussion.

Recall that, by rigidity, for each object $a\in \ca$ we have a {\it left} dual object $\overline{a}\in \ca$ with morphisms 

\[
{\rm coev}_{a}\in \Hom_\ca(\mathbb{1}, a\otimes \overline{a}) \quad \textrm{and} \quad {\rm ev}_{a}\in \Hom_\ca( \overline{a}\otimes a, \mathbb{1})
\]
represented by cups and caps, respectively. These satisfy the duality (or {\it zig-zag}) equations; see, for example, \cite{Mu1} for duality in graphical calculus. Similarly, $\overline{a}\in \ca$ has a left dual $\overline{\overline{a}}$ with morphisms:
\[
{\rm coev}_{\overline{a}}\in\Hom_ \ca(\mathbb{1}, \overline{a}\otimes \overline{\overline{a}}) \quad \textrm{and} \quad {\rm ev}_{\overline{a}}\in \Hom_\ca(\overline{\overline{a}}\otimes \overline{a},  \mathbb{1})
\]
Using the pivotal structure, these morphisms make $\overline{a}$ also a {\it right} dual of $a$. They are basis elements for the transpose spaces of the left dual morphism spaces:
\[
{\rm coev}_{\overline{a}}\in\Hom_ \ca(\mathbb{1}, \overline{a}\otimes a) \quad \textrm{and} \quad {\rm ev}_{\overline{a}}\in \Hom_\ca(a\otimes \overline{a},  \mathbb{1})
\]
Now these four morphisms are represented graphically as {\it oriented} caps and cups.

The left and right duality can be extended to a monoidal equivalence of categories between $\ca$ and $\ca^{op}$. For $f\in \Hom_\ca(a, b)$ we have the {\it left transpose} $\overline{f}^{l}\in \Hom_\ca(\overline{b},\overline{a})$ which is given by:
\[
\overline{f}^{l}:=({\rm ev}_{b}\otimes 1_{\overline{a}})\circ(1_{\overline{b}}\otimes f\otimes 1_{\overline{a}})\circ (1_{\overline{b}}\otimes {\rm coev}_{a})
\]
In diagrams this is left rotation of $f$ by $\pi$. Likewise, we define the {\it right transpose} $\overline{f}^{r}\in \Hom_\ca(\overline{b},\overline{a})$ via right rotation by $\pi$ as follows:

\[
\overline{f}^{r}:=(1_{\overline{a}}\otimes {\rm ev}_{\overline{b}})\circ (1_{\overline{a}}\otimes f \otimes 1_{\overline{b}})\circ({\rm coev}_{\overline{a}}\otimes 1_{\overline{b}})
\]

The pivotal structure gives a monoidal natural isomorphism between left and right dualities, and the strictness of the pivotal structures means $\overline{f}^{l}=\overline{f}^{r}$, hence we can unambiguously define $\overline{f}=\overline{f}^{l}=\overline{f}^{r}$.  (Using this fact it is easy to check that the duality functor $\overline{f}$ satisfies $\overline{\overline{f}}=f$.) Strictness of the pivotal structure is then represented graphically as:

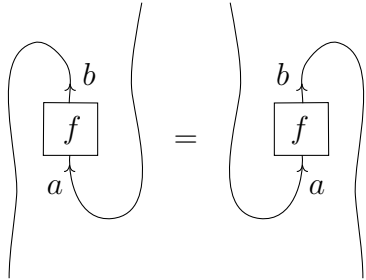
\[
\begin{tikzpicture}[yscale=-1,xscale=-1,scale=0.05,
baseline={([yshift=-.5ex]current bounding box.center)}]
\draw [fill=none,draw=black] (287.35mm,428.09mm)
-- ++(0.00mm,139.40mm)
-- ++(144.45mm,0.00mm)
-- ++(-1.01mm,-140.41mm)
-- cycle
;
\draw [fill=none,draw=black] (362.99mm,428.63mm)
.. controls ++(-0.64mm,-32.26mm) and ++(-18.44mm,29.90mm) .. ++(8.22mm,-105.73mm)
.. controls ++(17.25mm,-27.98mm) and ++(-30.42mm,-7.40mm) .. ++(83.41mm,-55.72mm)
.. controls ++(127.66mm,31.05mm) and ++(-7.76mm,-103.76mm) .. ++(48.17mm,407.19mm)
.. controls ++(5.45mm,72.89mm) and ++(0.00mm,-72.80mm) .. ++(19.97mm,220.02mm)
;
\draw [fill=none,draw=black] (361.43mm,566.93mm)
.. controls (366.69mm,633.62mm) and (337.55mm,745.39mm) .. (248.57mm,735.51mm)
.. controls (109.15mm,720.01mm) and (192.32mm,465.53mm) .. (197.14mm,388.36mm)
.. controls (201.62mm,316.68mm) and (180.86mm,222.79mm) .. (170.00mm,161.22mm)
;
\draw [fill=none,draw=black] (345.64mm,394.52mm)
.. controls ++(9.73mm,0.79mm) and ++(-1.02mm,7.65mm) .. ++(15.15mm,-17.43mm)
;
\draw [fill=none,draw=black] (375.24mm,394.53mm)
.. controls ++(-9.41mm,0.79mm) and ++(0.99mm,7.66mm) .. ++(-14.66mm,-17.43mm)
;
\draw [fill=none,draw=black] (347.06mm,610.24mm)
.. controls ++(9.73mm,0.79mm) and ++(-1.02mm,7.65mm) .. ++(15.15mm,-17.43mm)
;
\draw [fill=none,draw=black] (376.67mm,610.24mm)
.. controls ++(-9.41mm,0.79mm) and ++(0.99mm,7.66mm) .. ++(-14.66mm,-17.43mm)
;
\node [black] at (400.43mm,652.36mm) { $a$ };
\node [black] at (310.14mm,350.51mm) { $b$ };
\node [black] at (355.14mm,500.08mm) { $f$ };
\end{tikzpicture}=\begin{tikzpicture}[yscale=-1,scale=0.05,
baseline={([yshift=-.5ex]current bounding box.center)}]
\draw [fill=none,draw=black] (287.35mm,428.09mm)
-- ++(0.00mm,139.40mm)
-- ++(144.45mm,0.00mm)
-- ++(-1.01mm,-140.41mm)
-- cycle
;
\draw [fill=none,draw=black] (362.99mm,428.63mm)
.. controls ++(-0.64mm,-32.26mm) and ++(-18.44mm,29.90mm) .. ++(8.22mm,-105.73mm)
.. controls ++(17.25mm,-27.98mm) and ++(-30.42mm,-7.40mm) .. ++(83.41mm,-55.72mm)
.. controls ++(127.66mm,31.05mm) and ++(-7.76mm,-103.76mm) .. ++(48.17mm,407.19mm)
.. controls ++(5.45mm,72.89mm) and ++(0.00mm,-72.80mm) .. ++(19.97mm,220.02mm)
;
\draw [fill=none,draw=black] (361.43mm,566.93mm)
.. controls (366.69mm,633.62mm) and (337.55mm,745.39mm) .. (248.57mm,735.51mm)
.. controls (109.15mm,720.01mm) and (192.32mm,465.53mm) .. (197.14mm,388.36mm)
.. controls (201.62mm,316.68mm) and (180.86mm,222.79mm) .. (170.00mm,161.22mm)
;
\draw [fill=none,draw=black] (345.64mm,394.52mm)
.. controls ++(9.73mm,0.79mm) and ++(-1.02mm,7.65mm) .. ++(15.15mm,-17.43mm)
;
\draw [fill=none,draw=black] (375.24mm,394.53mm)
.. controls ++(-9.41mm,0.79mm) and ++(0.99mm,7.66mm) .. ++(-14.66mm,-17.43mm)
;
\draw [fill=none,draw=black] (347.06mm,610.24mm)
.. controls ++(9.73mm,0.79mm) and ++(-1.02mm,7.65mm) .. ++(15.15mm,-17.43mm)
;
\draw [fill=none,draw=black] (376.67mm,610.24mm)
.. controls ++(-9.41mm,0.79mm) and ++(0.99mm,7.66mm) .. ++(-14.66mm,-17.43mm)
;
\node [black] at (400.43mm,652.36mm) { $a$ };
\node [black] at (310.14mm,350.51mm) { $b$ };
\node [black] at (355.14mm,500.08mm) { $f$ };
\end{tikzpicture}
\]

In a pivotal fusion category one may define the left and right {\bf pivotal} (or {\bf quantum}) {\bf traces}  of an endomorphism $f\in \Hom_\ca (a,a)$ of an object $a\in \ca$:

\[
{\rm tr}_{l}(f)={\rm ev}_{a}\circ (1_{\overline{a}}\otimes f)\circ {\rm coev}_{\overline{a}}
\quad \textrm{and} \quad 
{\rm tr}_{r}(f)={\rm ev}_{\overline{a}}\circ (f\otimes 1_{\overline{a}})\circ {\rm coev}_{a}
\]

\begin{defi}
A pivotal fusion category is called {\bf spherical} if the left and right quantum traces agree for all objects $a \in \ca$ and morphisms $f\in \Hom_\ca(a,a)$. This condition is depicted as the following equality of diagrams:
\[
\begin{tikzpicture}[yscale=-1,scale=0.05,
baseline={([yshift=-.5ex]current bounding box.center)}]
\draw [fill=none,draw=black] (434.29mm,345.22mm)
-- ++(0.00mm,127.14mm)
-- ++(120.00mm,0.00mm)
-- ++(0.00mm,-125.71mm)
-- cycle
;
\draw [fill=none,draw=black] (500.00mm,346.47mm)
.. controls (495.99mm,277.20mm) and (522.68mm,198.48mm) .. (485.71mm,133.79mm)
.. controls (448.41mm,68.50mm) and (329.18mm,46.52mm) .. (304.29mm,129.51mm)
.. controls ++(-34.34mm,114.46mm) and ++(-43.47mm,-115.92mm) .. ++(20.00mm,532.86mm)
.. controls (336.60mm,695.20mm) and (363.82mm,725.37mm) .. (400.00mm,730.93mm)
.. controls (546.99mm,753.55mm) and (532.16mm,556.02mm) .. (515.30mm,471.70mm)
;
\draw [fill=none,draw=black] (487.73mm,224.40mm)
.. controls ++(12.49mm,3.46mm) and ++(-0.59mm,11.53mm) .. ++(17.53mm,-21.63mm)
;
\draw [fill=none,draw=black] (487.73mm,224.40mm)
.. controls ++(12.49mm,3.46mm) and ++(-0.59mm,11.53mm) .. ++(17.53mm,-21.63mm)
;
\draw [fill=none,draw=black] (521.95mm,224.79mm)
.. controls ++(-11.87mm,3.46mm) and ++(0.56mm,11.53mm) .. ++(-16.66mm,-21.64mm)
;
\draw [fill=none,draw=black] (508.23mm,599.88mm)
.. controls ++(12.49mm,3.46mm) and ++(-0.59mm,11.53mm) .. ++(17.53mm,-21.63mm)
;
\draw [fill=none,draw=black] (508.23mm,599.88mm)
.. controls ++(12.49mm,3.46mm) and ++(-0.59mm,11.53mm) .. ++(17.53mm,-21.63mm)
;
\draw [fill=none,draw=black] (542.45mm,600.27mm)
.. controls ++(-11.87mm,3.46mm) and ++(0.56mm,11.53mm) .. ++(-16.66mm,-21.64mm)
;
\draw [fill=none,draw=black] (266.54mm,346.30mm)
.. controls ++(12.49mm,-3.46mm) and ++(-0.59mm,-11.53mm) .. ++(17.53mm,21.63mm)
;
\draw [fill=none,draw=black] (266.54mm,346.30mm)
.. controls ++(12.49mm,-3.46mm) and ++(-0.59mm,-11.53mm) .. ++(17.53mm,21.63mm)
;
\draw [fill=none,draw=black] (300.75mm,345.91mm)
.. controls ++(-11.87mm,-3.46mm) and ++(0.56mm,-11.53mm) .. ++(-16.66mm,21.64mm)
;
\node [black] at (217.39mm,361.32mm) { $a$ };
\node [black] at (494.87mm,412.13mm) { $f$ };
\end{tikzpicture}=\begin{tikzpicture}[yscale=-1,xscale=-1,scale=0.05,
baseline={([yshift=-.5ex]current bounding box.center)}]
\draw [fill=none,draw=black] (434.29mm,345.22mm)
-- ++(0.00mm,127.14mm)
-- ++(120.00mm,0.00mm)
-- ++(0.00mm,-125.71mm)
-- cycle
;
\draw [fill=none,draw=black] (500.00mm,346.47mm)
.. controls (495.99mm,277.20mm) and (522.68mm,198.48mm) .. (485.71mm,133.79mm)
.. controls (448.41mm,68.50mm) and (329.18mm,46.52mm) .. (304.29mm,129.51mm)
.. controls ++(-34.34mm,114.46mm) and ++(-43.47mm,-115.92mm) .. ++(20.00mm,532.86mm)
.. controls (336.60mm,695.20mm) and (363.82mm,725.37mm) .. (400.00mm,730.93mm)
.. controls (546.99mm,753.55mm) and (532.16mm,556.02mm) .. (515.30mm,471.70mm)
;
\draw [fill=none,draw=black] (487.73mm,224.40mm)
.. controls ++(12.49mm,3.46mm) and ++(-0.59mm,11.53mm) .. ++(17.53mm,-21.63mm)
;
\draw [fill=none,draw=black] (487.73mm,224.40mm)
.. controls ++(12.49mm,3.46mm) and ++(-0.59mm,11.53mm) .. ++(17.53mm,-21.63mm)
;
\draw [fill=none,draw=black] (521.95mm,224.79mm)
.. controls ++(-11.87mm,3.46mm) and ++(0.56mm,11.53mm) .. ++(-16.66mm,-21.64mm)
;
\draw [fill=none,draw=black] (508.23mm,599.88mm)
.. controls ++(12.49mm,3.46mm) and ++(-0.59mm,11.53mm) .. ++(17.53mm,-21.63mm)
;
\draw [fill=none,draw=black] (508.23mm,599.88mm)
.. controls ++(12.49mm,3.46mm) and ++(-0.59mm,11.53mm) .. ++(17.53mm,-21.63mm)
;
\draw [fill=none,draw=black] (542.45mm,600.27mm)
.. controls ++(-11.87mm,3.46mm) and ++(0.56mm,11.53mm) .. ++(-16.66mm,-21.64mm)
;
\draw [fill=none,draw=black] (266.54mm,346.30mm)
.. controls ++(12.49mm,-3.46mm) and ++(-0.59mm,-11.53mm) .. ++(17.53mm,21.63mm)
;
\draw [fill=none,draw=black] (266.54mm,346.30mm)
.. controls ++(12.49mm,-3.46mm) and ++(-0.59mm,-11.53mm) .. ++(17.53mm,21.63mm)
;
\draw [fill=none,draw=black] (300.75mm,345.91mm)
.. controls ++(-11.87mm,-3.46mm) and ++(0.56mm,-11.53mm) .. ++(-16.66mm,21.64mm)
;
\node [black] at (217.39mm,361.32mm) { $a$ };
\node [black] at (494.87mm,412.13mm) { $f$ };
\end{tikzpicture}
\]
\end{defi}

We are also now able to define the {\bf quantum dimension} $d_a$ of an object $a\in \ca$ by letting $d_a := {\rm tr}(1_a)$, the trace of the identity morphism $1_a\in \Hom_\ca(a,a)$. Then the {\bf global dimension} $D_\ca$ of $\ca$ is defined to be 
\[
D_\ca = \sum_{a\in \Irr(\ca)}d^{2}_{a}
\]

\subsection{Ribbon categories} A {\bf braided} tensor category has, for each pair of objects, an isomorphism $c_{a,b}\in \Hom_\ca(a\otimes b, b\otimes a)$ called the braiding. This is a natural isomorphism between the bifunctors $\otimes$ and $\otimes^{\rm op}$ and it satisfies the braid (or hexagon) relations.  The morphisms $c_{a,b}$ and $c^{-1}_{a,b}$ are represented, respectively, by the pictures
\[
\begin{tikzpicture}[yscale=-1,scale=0.05,
baseline={([yshift=-.5ex]current bounding box.center)}]
\draw [fill=none,draw=black] (134.29mm,666.65mm)
.. controls (440.00mm,396.65mm) and (440.00mm,396.65mm) .. (440.00mm,396.65mm)
;
\draw [fill=none,draw=black] (437.14mm,665.22mm)
.. controls (304.29mm,542.36mm) and (304.29mm,542.36mm) .. (304.29mm,542.36mm)
;
\draw [fill=none,draw=black] (274.62mm,515.70mm)
-- (155.24mm,398.88mm)
;
\draw [fill=none,draw=black] (175.51mm,438.45mm)
.. controls ++(4.44mm,-7.03mm) and ++(4.29mm,6.74mm) .. ++(-0.90mm,-20.70mm)
;
\draw [fill=none,draw=black] (195.09mm,417.45mm)
.. controls ++(-6.79mm,4.80mm) and ++(6.95mm,3.93mm) .. ++(-20.72mm,0.19mm)
;
\draw [fill=none,draw=black] (388.84mm,424.05mm)
.. controls ++(7.01mm,4.48mm) and ++(-6.76mm,4.25mm) .. ++(20.70mm,-0.78mm)
;
\draw [fill=none,draw=black] (409.73mm,443.76mm)
.. controls ++(-4.76mm,-6.82mm) and ++(-3.97mm,6.93mm) .. ++(-0.06mm,-20.72mm)
;
\node [black] at (124.29mm,628.08mm) { $a$ };
\node [black] at (447.14mm,605.22mm) { $b$ };
\end{tikzpicture} \quad \textrm{and} \quad \begin{tikzpicture}[yscale=-1,xscale=-1,scale=0.05,
baseline={([yshift=-.5ex]current bounding box.center)}]
\draw [fill=none,draw=black] (134.29mm,666.65mm)
.. controls (440.00mm,396.65mm) and (440.00mm,396.65mm) .. (440.00mm,396.65mm)
;
\draw [fill=none,draw=black] (437.14mm,665.22mm)
.. controls (304.29mm,542.36mm) and (304.29mm,542.36mm) .. (304.29mm,542.36mm)
;
\draw [fill=none,draw=black] (274.62mm,515.70mm)
-- (155.24mm,398.88mm)
;
\draw [fill=none,draw=black] (175.51mm,438.45mm)
.. controls ++(4.44mm,-7.03mm) and ++(4.29mm,6.74mm) .. ++(-0.90mm,-20.70mm)
;
\draw [fill=none,draw=black] (195.09mm,417.45mm)
.. controls ++(-6.79mm,4.80mm) and ++(6.95mm,3.93mm) .. ++(-20.72mm,0.19mm)
;
\draw [fill=none,draw=black] (388.84mm,424.05mm)
.. controls ++(7.01mm,4.48mm) and ++(-6.76mm,4.25mm) .. ++(20.70mm,-0.78mm)
;
\draw [fill=none,draw=black] (409.73mm,443.76mm)
.. controls ++(-4.76mm,-6.82mm) and ++(-3.97mm,6.93mm) .. ++(-0.06mm,-20.72mm)
;
\node [black] at (104.29mm,628.08mm) { $b$ };
\node [black] at (447.14mm,605.22mm) { $a$ };
\end{tikzpicture}
\]

\begin{defi}
A {\bf ribbon} fusion category is a {\it spherical} and {\it braided} fusion category.
\end{defi}
\noindent The ribbon structure is a natural isomorphism of the identity functor given by  isomorphisms $\theta_{a}\in \Hom_\ca(a,a)$ called {\bf twists} defined as follows:
\[ \theta_{a}:= (1_{a}\otimes {\rm ev}_{a})\circ c_{a,a}\circ(1_{a}\otimes {\rm coev}_{a})=\begin{tikzpicture}[yscale=-1,scale=0.05,
baseline={([yshift=-.5ex]current bounding box.center)}]
\draw [fill=none,draw=black] (204.29mm,705.22mm)
.. controls ++(21.33mm,-124.36mm) and ++(-173.92mm,43.48mm) .. ++(101.43mm,-390.00mm)
.. controls ++(12.93mm,-3.23mm) and ++(-8.21mm,-16.42mm) .. ++(60.00mm,10.00mm)
.. controls ++(38.90mm,77.80mm) and ++(33.89mm,79.08mm) .. ++(-110.00mm,42.86mm)
;
\draw [fill=none,draw=black] (234.29mm,319.51mm)
.. controls (204.35mm,247.47mm) and (183.81mm,162.79mm) .. (188.57mm,86.65mm)
;
\node [black,font=\sffamily] at (255.71mm,660.93mm) { $a$ };
\draw [fill=none,draw=black] (170.16mm,499.26mm)
.. controls ++(14.10mm,4.50mm) and ++(-3.69mm,14.51mm) .. ++(27.87mm,-27.88mm)
;
\draw [fill=none,draw=black] (224.55mm,499.76mm)
.. controls ++(-13.85mm,4.50mm) and ++(3.63mm,14.51mm) .. ++(-27.37mm,-27.89mm)
;
\end{tikzpicture}\]
In particular, if $a$ is simple, $\theta_{a}$ is a scalar multiple of the identity morphism $1_{a}$. (As is common practice we often abuse notation and write $\theta_{a}$ for this scalar.)  The twist isomorphisms satisfy the {\it ribbon axiom}:
$$\theta_{a\otimes b}=c_{b,a}\circ c_{a,b}\circ (\theta_{a}\otimes \theta_{b})$$

\subsection{Modularity and the Drinfel'd center}

Given a spherical fusion category $\ca$ one can produce a canonical ribbon fusion category $Z(\ca)$ called the {\bf Drinfel'd center}. Briefly, $Z(\ca)$ is obtained as the category of $\ca \mymathhyphen \ca$ bimodule endofunctors on $\ca$ viewed as a $\ca \mymathhyphen \ca$ bimodule category. The objects of $Z(\ca)$ are given by a pair of an object from $\ca$ and a {\bf half-braiding}. That is, the bimodule endofunctor $F$ corresponds to the object $F(\mathbb{1}) \in \ca$ along with half-braiding natural isomorphisms $e_{F,a}:F(\mathbb{1})\otimes a\rightarrow a\otimes F(\mathbb{1})$ for all objects $a\in \ca$ coming from the bimodule structure morphisms of $F$. The Drinfel'd center is completely determined by the half-braidings. See \cite[Definition 7.13.1]{ENOG} for details.

Most significantly for our purposes, the Drinfel'd center of a spherical fusion category is always \cite{Mu2} a {\it modular} category:
\begin{defi}

A ribbon fusion category is {\bf modular} if the braiding is {\it non-degenerate}, i.e. if the {\bf S-matrix} is {\it invertible}. The S-matrix for a ribbon fusion category $\da$ with $n$ simple objects is the $n\times n$ matrix $S:= \frac{1}{\sqrt{D_\da}} \hat{S}$ where:

\[\hat{S}_{a,b}:={\rm tr} \; (c_{b,\overline{a}}\circ c_{\overline{a},b})=\begin{tikzpicture}[yscale=-1,scale=0.03,baseline={([yshift=-.5ex]current bounding box.center)}]
\begin{scope}[shift={(0.00mm,719.29mm)}]
\draw [fill=none,draw=black] (554.29mm,-188.47mm)
.. controls (420.64mm,-186.67mm) and (250.32mm,-53.78mm) .. (251.45mm,137.22mm)
;
\draw [fill=none,draw=black] (554.29mm,459.53mm)
.. controls (420.64mm,457.73mm) and (250.32mm,324.83mm) .. (251.45mm,133.84mm)
;
\draw [fill=none,draw=black] (552.93mm,459.53mm)
.. controls (686.58mm,457.73mm) and (856.90mm,324.83mm) .. (855.76mm,133.84mm)
;
\draw [fill=none,draw=black] (1012.93mm,439.53mm)
.. controls ++(133.65mm,-1.80mm) and ++(1.13mm,190.99mm) .. ++(302.83mm,-325.69mm)
;
\draw [fill=none,draw=black] (1012.93mm,-208.47mm)
.. controls ++(133.65mm,1.80mm) and ++(1.13mm,-190.99mm) .. ++(302.83mm,325.69mm)
;
\draw [fill=none,draw=black] (1014.29mm,-208.47mm)
.. controls (880.64mm,-206.67mm) and (710.32mm,-73.78mm) .. (711.45mm,117.22mm)
;
\draw [fill=none,draw=black] (553.93mm,-188.53mm)
.. controls ++(113.03mm,4.69mm) and ++(-56.58mm,-38.31mm) .. ++(209.13mm,67.76mm)
;
\draw [fill=none,draw=black] (819.91mm,-65.41mm)
.. controls ++(62.24mm,81.03mm) and ++(-9.04mm,-65.83mm) .. ++(35.81mm,199.92mm)
;
\draw [fill=none,draw=black] (711.43mm,118.79mm)
.. controls ++(7.70mm,62.02mm) and ++(-19.31mm,-48.27mm) .. ++(61.62mm,173.12mm)
;
\draw [fill=none,draw=black] (801.60mm,353.58mm)
.. controls ++(47.17mm,70.41mm) and ++(-93.51mm,5.75mm) .. ++(211.26mm,85.93mm)
;
\node [black] at (314.29mm,-143.35mm) { $a$ };
\node [black] at (1300mm,-161.92mm) { $b$ };
\draw [fill=none,draw=black] (202.86mm,143.79mm)
.. controls ++(27.08mm,2.45mm) and ++(-2.59mm,20.73mm) .. ++(48.57mm,-40.00mm)
;
\draw [fill=none,draw=black] (299.88mm,143.79mm)
.. controls ++(-27.08mm,2.45mm) and ++(2.59mm,20.73mm) .. ++(-48.57mm,-40.00mm)
;
\draw [fill=none,draw=black] (806.90mm,102.40mm)
.. controls ++(27.08mm,-2.45mm) and ++(-2.59mm,-20.73mm) .. ++(48.57mm,40.00mm)
;
\draw [fill=none,draw=black] (903.92mm,102.40mm)
.. controls ++(-27.08mm,-2.45mm) and ++(2.59mm,-20.73mm) .. ++(-48.57mm,40.00mm)
;
\draw [fill=none,draw=black] (664.37mm,123.79mm)
.. controls ++(27.08mm,2.45mm) and ++(-2.59mm,20.73mm) .. ++(48.57mm,-40.00mm)
;
\draw [fill=none,draw=black] (761.39mm,123.79mm)
.. controls (734.31mm,126.24mm) and (715.41mm,104.52mm) .. (712.82mm,83.79mm)
;
\draw [fill=none,draw=black] (1267.25mm,84.18mm)
.. controls ++(27.08mm,-2.45mm) and ++(-2.59mm,-20.73mm) .. ++(48.57mm,40.00mm)
;
\draw [fill=none,draw=black] (1364.28mm,84.18mm)
.. controls ++(-27.08mm,-2.45mm) and ++(2.59mm,-20.73mm) .. ++(-48.57mm,40.00mm)
;
\end{scope}
\end{tikzpicture}\]

\end{defi}    
The \textit{Verlinde formula} allows the fusion rules of a modular category $\da$ to be expressed in terms of entries of its $S$-matrix:
$$N^{a}_{c,d}=\sum_{e\in \text{Irr}(\ca)} \frac{S_{c,e} S_{d,e} \overline{S}_{a,e}}{S_{\mathbb{1},e}}$$
The $S$-matrix along with the $n\times n$ matrix $T=\hat{T}$ where $\hat{T}_{a,b}:=\delta_{a,b} \theta_{a}$ are together referred to as the {\bf modular data}. By \cite[Proposition 5.7]{NS2} the entries of $\hat{S}$ and $\hat{T}$ are cyclotomic integers in $\mathbb{Q}( e^{2 \pi i / M} )$. Vafa's theorem asserts that all twists are roots of unity, hence $T$ is finite order; the number $M:=ord(T)$ is called the {\bf conductor} of the category. 

The {\bf charge conjugation} matrix is the $n\times n$ matrix $C=\delta_{\overline{a},b}$, and the {\bf central charge} is the positive square root of the ratio of the {\bf Gauss sums} for $\da$: 
\[
\xi=\frac{1}{\sqrt{D_\da}}\sum_{a} \theta_{a} d^{2}_{a}
\]
If $\da$ is modular then this is always defined and a root of unity.  These matrices satisfy the following relations:
\[(ST)^{3}=\xi S^{2},\ S^{2}=C,\  C^{2}=1,\  CS=SC,\ CT=TC. \]
Which implies that
\[
\left( \begin{array}{cc} 0 & -1\\1 & 0 \\ \end{array}\right) \rightarrow S \quad \left( \begin{array}{cc} 1 & 1 \\ 0 & 1\\ \end{array}\right)\rightarrow T
\]
is a projective representations of $SL_{2}(\Z)$.    See \cite[Chapter 3]{BK} for all of the above results.  It was proved in \cite{NS2} that this representation of $SL_{2}(\Z)$ factors through $SL_{2}(\Z_{M})$.   

\subsection{Induction to the center}

Consider the forgetful functor $Z(\ca)\rightarrow \ca$. This is given by forgetting the half-braiding, i.e. by sending the bimodule endofunctor $F\in Z(\ca)$ to the object $F(\mathbb{1})\in \ca$. M\"uger showed that the forgetful functor has a left adjoint $\mathcal{I}:\ca \to Z(\ca)$ called the {\bf central induction functor}. 

If we start with a ribbon fusion category $\da$ then where it ``sits inside'' its Drinfel'd center $Z(\da)$ is particularly well understood. Define a new ribbon fusion category $\tilde{\da}$ such that  $\tilde{\da} = \da$ as tensor categories but with a new braiding :
\[ 
\tilde{c}_{a,b}:=c^{-1}_{b,a}. 
\]
which gives the following modular data:
\[
\tilde{S}_{a, b}=S^{-1}_{a,b}=S_{\overline{a}, b} \quad \textrm{and} \quad \tilde{T}_{a}=\theta_{\tilde{a}}=\theta^{-1}_{a}
\]

For any ribbon fusion category we always have a braided tensor functor 
\[
G: \da\boxtimes \tilde{\da}\rightarrow Z(\da)
\]
where $\boxtimes$ is the Deligne tensor product of abelian categories. If $\da$ is modular then by \cite[Theorem 7.10]{Mu2} $G$ is a {\it braided equivalence} of tensor categories. Using $G$, we can identify $Z(\da)$ with $\da \boxtimes \tilde{\da}$. Under this identification, the forgetful functor $Z(\da)\rightarrow \da$ is given by:
\[
X\boxtimes Y \mapsto X\otimes Y
\]
along with forgetting about the braiding.

Notice that $G|_{\da\boxtimes \mathbb{1}}$ is a tensor equivalence onto its image, hence any ribbon fusion category embeds as a full subcategory into a modular one.  In particular, this implies the $T$-matrix of a (not necessarilly modular) ribbon fusion category also has finite order, which we also call the conductor.

Thus, in the case that $\da$ is modular, we can recover the modular data for $Z(\da)$ from the modular data of $\da$ and the structure of the forgetful functor $Z(\da)\rightarrow \da$. Since the simple objects of $Z(\da)$ are indexed by equivalence classes of $a\boxtimes \tilde{b}$ we have the following identities: 
\[
 S_{a\boxtimes  \tilde{b}, c\boxtimes  \tilde{d}}=S_{a,c}\tilde{S}_{b,d}=S_{a,c} S_{\overline{b}, d}\quad
\]
\[\theta_{a\boxtimes  \tilde{b}}=\frac{\theta_{a}}{\theta_{b}}\]
These are a consequence of M\"uger's result mentioned above; their computation can be found in the proof  of \cite[Lemma 6.2]{NS2} at equations (6.5) and (6.6).

\section{Rotation eigenvalues} \label{sec:generalized_rotation}

From here on out let $\ca$ be a spherical fusion category.

\subsection{Generalized rotation and GFS indicators}

For a spherical fusion category $\ca$ recall that objects in the Drinfel'd center $Z(\ca)$ are pairs of objects $b\in \ca$ with a half-braiding natural isomorphism $e_{b, \cdot}:b\otimes \cdot \to \cdot \otimes b$.

\begin{defi}\cite[Section 2]{NS2}
 Let $\ca$ be a spherical fusion category. For objects $a\in \ca$ and $b\in Z(\ca)$ and positive integer $n\in \N$ we define the {\bf generalized rotation operator} 
 \[
 \rho^{b}_{n,a}: \Hom_\ca(b, a^{\otimes n})\rightarrow \Hom_\ca(b, a^{\otimes n})
 \]
given by:
\[
f \mapsto ({\rm ev}_{a}\otimes 1_{a^{\otimes n}})\circ (1_{\overline{a}}\otimes f\otimes 1_{a})\circ(e_{b,\overline{a}}\otimes 1_{a})\circ (1_{b}\otimes {\rm coev}_{a})
\]
which is expressed diagrammatically by:
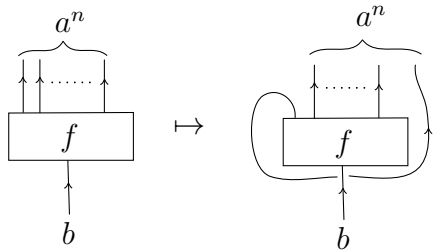
\[
\begin{tikzpicture}[scale=0.05,yscale=-1,
baseline={([yshift=-.5ex]current bounding box.center)}]
\draw [fill=none,draw=black] (191.43mm,372.36mm)
-- ++(0.00mm,125.71mm)
-- ++(331.43mm,0.00mm)
-- ++(0.00mm,-127.14mm)
-- cycle
;
\draw [fill=none,draw=black] (230.00mm,371.51mm)
-- ++(0.00mm,-141.43mm)
-- ++(0.00mm,0.00mm)
;
\draw [fill=none,draw=black] (274.29mm,372.36mm)
.. controls ++(0.00mm,-142.86mm) and ++(0.00mm,0.00mm) .. ++(0.00mm,-142.86mm)
;
\draw [fill=none,draw=black] (445.71mm,371.51mm)
.. controls ++(0.00mm,-141.43mm) and ++(0.00mm,0.00mm) .. ++(0.00mm,-141.43mm)
;
\draw [fill=black,draw=black] (307.14mm,290.93mm) circle (0.92mm) ;
\draw [fill=black,draw=black] (327.14mm,290.93mm) circle (0.92mm) ;
\draw [fill=black,draw=black] (347.14mm,290.93mm) circle (0.92mm) ;
\draw [fill=black,draw=black] (367.14mm,290.93mm) circle (0.92mm) ;
\draw [fill=black,draw=black] (387.14mm,290.93mm) circle (0.92mm) ;
\draw [fill=black,draw=black] (407.14mm,290.93mm) circle (0.92mm) ;
\draw [fill=none,draw=black] (348.56mm,498.08mm)
.. controls ++(4.30mm,144.27mm) and ++(0.00mm,0.00mm) .. ++(4.30mm,144.27mm)
;
\draw [fill=none,draw=black] (217.69mm,291.22mm)
.. controls ++(7.15mm,-0.48mm) and ++(-1.47mm,1.81mm) .. ++(11.50mm,-14.16mm)
;
\draw [fill=none,draw=black] (238.90mm,291.21mm)
.. controls ++(-5.75mm,-2.70mm) and ++(1.18mm,1.66mm) .. ++(-10.10mm,-14.14mm)
;
\draw [fill=none,draw=black] (263.15mm,291.22mm)
.. controls ++(7.15mm,-0.48mm) and ++(-1.47mm,1.81mm) .. ++(11.50mm,-14.16mm)
;
\draw [fill=none,draw=black] (284.36mm,291.21mm)
.. controls ++(-5.75mm,-2.70mm) and ++(1.18mm,1.66mm) .. ++(-10.10mm,-14.14mm)
;
\draw [fill=none,draw=black] (434.88mm,296.27mm)
.. controls ++(7.15mm,-0.48mm) and ++(-1.47mm,1.81mm) .. ++(11.50mm,-14.16mm)
;
\draw [fill=none,draw=black] (456.08mm,296.26mm)
.. controls ++(-5.75mm,-2.70mm) and ++(1.18mm,1.66mm) .. ++(-10.10mm,-14.14mm)
;
\draw [fill=none,draw=black] (338.91mm,572.04mm)
.. controls ++(7.15mm,-0.48mm) and ++(-1.47mm,1.81mm) .. ++(11.50mm,-14.16mm)
;
\draw [fill=none,draw=black] (360.12mm,572.03mm)
.. controls ++(-5.75mm,-2.70mm) and ++(1.18mm,1.66mm) .. ++(-10.10mm,-14.14mm)
;
\draw [fill=none,draw=black] (336.21mm,180.89mm)
.. controls ++(-26.39mm,48.12mm) and ++(36.84mm,-32.57mm) .. ++(-119.07mm,39.32mm)
;
\draw [fill=none,draw=black] (336.16mm,180.93mm)
.. controls ++(25.89mm,47.21mm) and ++(-36.14mm,-31.96mm) .. ++(116.81mm,38.58mm)
;
\node [black] at (355.71mm,130.22mm) { $a^n$ };
\node [black] at (352.86mm,439.51mm) { $f$ };
\node [black] at (350.57mm,690.93mm) { $b$ };
\end{tikzpicture}
\quad \mapsto \quad
\begin{tikzpicture}[yscale=-1,scale=0.05,
baseline={([yshift=-.5ex]current bounding box.center)}]]
\draw [fill=none,draw=black] (191.43mm,372.36mm)
-- ++(0.00mm,125.71mm)
-- ++(331.43mm,0.00mm)
-- ++(0.00mm,-127.14mm)
-- cycle
;
\draw [fill=none,draw=black] (274.29mm,372.36mm)
.. controls ++(0.00mm,-142.86mm) and ++(0.00mm,0.00mm) .. ++(0.00mm,-142.86mm)
;
\draw [fill=none,draw=black] (445.71mm,371.51mm)
.. controls ++(0.00mm,-141.43mm) and ++(0.00mm,0.00mm) .. ++(0.00mm,-141.43mm)
;
\draw [fill=black,draw=black] (307.14mm,290.93mm) circle (0.92mm) ;
\draw [fill=black,draw=black] (327.14mm,290.93mm) circle (0.92mm) ;
\draw [fill=black,draw=black] (347.14mm,290.93mm) circle (0.92mm) ;
\draw [fill=black,draw=black] (367.14mm,290.93mm) circle (0.92mm) ;
\draw [fill=black,draw=black] (387.14mm,290.93mm) circle (0.92mm) ;
\draw [fill=black,draw=black] (407.14mm,290.93mm) circle (0.92mm) ;
\draw [fill=none,draw=black] (348.56mm,498.08mm)
.. controls ++(4.30mm,144.27mm) and ++(0.00mm,0.00mm) .. ++(4.30mm,144.27mm)
;
\draw [fill=none,draw=black] (566.27mm,412.65mm)
.. controls ++(7.15mm,-0.48mm) and ++(-1.47mm,1.81mm) .. ++(11.50mm,-14.16mm)
;
\draw [fill=none,draw=black] (587.47mm,412.64mm)
.. controls ++(-5.75mm,-2.70mm) and ++(1.18mm,1.66mm) .. ++(-10.10mm,-14.14mm)
;
\draw [fill=none,draw=black] (263.15mm,291.22mm)
.. controls ++(7.15mm,-0.48mm) and ++(-1.47mm,1.81mm) .. ++(11.50mm,-14.16mm)
;
\draw [fill=none,draw=black] (284.36mm,291.21mm)
.. controls ++(-5.75mm,-2.70mm) and ++(1.18mm,1.66mm) .. ++(-10.10mm,-14.14mm)
;
\draw [fill=none,draw=black] (434.88mm,296.27mm)
.. controls ++(7.15mm,-0.48mm) and ++(-1.47mm,1.81mm) .. ++(11.50mm,-14.16mm)
;
\draw [fill=none,draw=black] (456.08mm,296.26mm)
.. controls ++(-5.75mm,-2.70mm) and ++(1.18mm,1.66mm) .. ++(-10.10mm,-14.14mm)
;
\draw [fill=none,draw=black] (338.91mm,572.04mm)
.. controls ++(7.15mm,-0.48mm) and ++(-1.47mm,1.81mm) .. ++(11.50mm,-14.16mm)
;
\draw [fill=none,draw=black] (360.12mm,572.03mm)
.. controls ++(-5.75mm,-2.70mm) and ++(1.18mm,1.66mm) .. ++(-10.10mm,-14.14mm)
;
\draw [fill=none,draw=black] (402.17mm,162.92mm)
.. controls ++(-31.19mm,48.04mm) and ++(43.54mm,-32.52mm) .. ++(-140.71mm,39.26mm)
;
\draw [fill=none,draw=black] (402.01mm,162.99mm)
.. controls ++(34.79mm,47.07mm) and ++(-48.55mm,-31.86mm) .. ++(156.93mm,38.46mm)
;
\node [black] at (425.71mm,100.22mm) { $a^n$ };
\node [black] at (352.86mm,439.51mm) { $f$ };
\node [black] at (350.57mm,690.93mm) { $b$ };
\draw [fill=none,draw=black] (222.86mm,370.93mm)
.. controls ++(5.17mm,-48.26mm) and ++(40.24mm,-26.83mm) .. ++(-85.71mm,-58.57mm)
.. controls ++(-45.53mm,30.36mm) and ++(-59.72mm,-19.91mm) .. ++(20.00mm,214.29mm)
.. controls ++(55.50mm,18.50mm) and ++(-52.96mm,0.00mm) .. ++(178.57mm,1.43mm)
;
\draw [fill=none,draw=black] (365.71mm,525.22mm)
.. controls ++(151.54mm,14.33mm) and ++(28.08mm,196.54mm) .. ++(207.14mm,-177.14mm)
.. controls ++(-6.32mm,-44.25mm) and ++(-5.32mm,42.54mm) .. ++(-30.00mm,-118.57mm)
;
\end{tikzpicture} 
\]
\end{defi}

The following properties of the generalized rotation are immediate:
\begin{enumerate}
\item
$(\rho^{b}_{n,a})^{n}=\theta^{-1}_{b} id$. 
\item $\rho^{b}_{n,a}$ is diagonalizable
\end{enumerate}
The traces of the generalized rotation operators provide a very important invariant of spherical fusion categories:

\begin{defi} \cite[Definition 2.1]{NS2}
For an object $a\in \ca$ in a spherical fusion category we define the \textbf{generalized Frobenius-Schur (GFS) indicators} for $n>0$, $k\in \mathbb{Z}$, and $b\in Z(\ca)$:
\[ \nu^{b}_{n,k}(a):={\rm Tr}((\rho^{b}_{n,a})^{k}). \]
where ${\rm Tr}$ is the usual trace of a linear map.
\end{defi}
\begin{rem*}
A different version of these invariants were also studied in the Hopf algebra setting in \cite[Chapter 8]{SZ} where they are called {\bf equivariant FS indicators}. The two definitions coincide when $b\in \ca$ is the unit object.
\end{rem*}


The usual categorical Frobenius-Schur (FS) indicators are obtained in the special case where $b=\mathbb{1}$ and then $\nu_m(a) = \nu^{\mathbb{1}}_{m,1}(a)$. The following proposition enumerates some useful properties of GFS indicators:
\begin{prop}\cite{NS2}
Let $a,d \in \ca$ be objects in a spherical fusion category, let $b,c\in Z(\ca)$, and let $l,m \in \mathbb{N}$. Then:
\begin{enumerate}
\item 
$\nu^{b}_{m,l}(a\oplus d)=\nu^{b}_{m,l}(a)+\nu^{b}_{m,l}(d)$ if $GCD(m,l)=1$ \cite[Corollary 5.5]{NS2}
\item 
$\nu^{b\oplus c}_{m,l}(a)=\nu^{b}_{m,l}(a)+\nu^{c}_{m,l}(a)$ \cite[Remark 2.4]{NS2}
\item 
$\nu^{b}_{m,l}(a)=\nu^{b}_{\frac{m}{g},\frac{l}{g}} (a^{g})$, where $g=GCD(m,l)$ \cite[Proposition 2.8]{NS2}
\item 
If $b$ is irreducible in $Z(\ca)$, then $\nu^{b}_{m,l+km}(a)=\theta^{-k}_{b}\nu^{b}_{m,l}(a)$ \cite[Lemma 2.7]{NS2}
\end{enumerate}
\end{prop}

\subsection{Discrete Fourier transform}

Suppose we have a diagonalizable operator $L$ on a finite dimensional vector space, whose eigenvalues are $N^{th}$ roots of unity for some $N$.  We can recover the complete collection of eigenvalues of the operator knowing only the values ${\rm Tr}(L^{k})$ for integers $1\le k\le N$. 

Let $\zeta$ be a primitive $N^{th}$ root of unity. For a vector $x:=(x_{1}, \dots , x_{N})\in \C^{N}$, the {\bf discrete Fourier transform (DFT)}, written $F(x)=(F(x)_{1}, \dots F(x)_{N})\in \C^{N}$, is defined by 
\[F(x)_{k}=\sum^{N}_{m=1} x_{m}\zeta^{m k}.\]
The operator $F$ is invertible, and 
\[F^{-1}(X)_{k}:=\frac{1}{N} \sum^{N}_{m=1} X_{m}\zeta^{-m k}.\]
Let $K_{m}$ be the dimension of the eigenspace of $L$ corresponding to the root of unity $\zeta^{m}$.   Consider the vector $x=(K_{1},\ \dots, K_{N} )$.  
Then since ${\rm Tr}(L^{k})=\sum^{N}_{m=1} K_{m}\zeta^{m k}$, it follows that 
\[F(x)=({\rm Tr}(L), {\rm Tr}(L^{2}), \dots, {\rm Tr}(L^{N})).\]
This implies the $m^{\rm th}$ component of $F^{-1}({\rm Tr}(L), {\rm Tr}(L^{2}), \dots, {\rm Tr}(L^{N}))$ will be the multiplicity $K_{m}$ with which the eigenvalue $\zeta^{m}$ occurs for the operator $L$. Hence the inverse DFT applied to the vector of traces of the powers of $L$ will yield the eigenvalue multiplicity vector. 

\subsection{Eigenvalue multiplicities}

For any {\it simple} object $b\in Z(\ca)$ and {\it any} object $a\in \ca$ all the eigenvalues of the rotation operator $\rho^{b}_{n,a}$ are roots of $\theta^{-1}_{b}$, so re-normalizing by an $n^{th}$-root of $\theta_{b}$ gives an operator whose eigenvalues are all roots of unity.  Thus its eigenspaces are isotypic components for the corresponding cyclic group representation. 

Define the renormalized rotation operator $\kappa^{b}_{n,a}=\theta_{b}^{\frac{1}{n}} \rho^{b}_{n,a}$, where $\theta_{b}^{\frac{1}{n}} $ is an arbitrary (but fixed) root of $\theta_{b}$.  Since $(\kappa^{b}_{n,a})^{n}=1$, its eigenvalues are $n^{th}$ roots of unity, hence we can use the discrete Fourier transform to compute eigenvalue multiplicities of the renormalized rotation.  Note that $Tr((\kappa^{b}_{n,a})^{k})=\theta^{\frac{k}{n}}\nu^{b}_{n,k}(a)$.  Thus if $\omega$ is a $n^{th}$ root of unity, the multiplicity of $\omega$ as an eigenvalue of rotation $\kappa^{b}_{n,a}$  is given by
 \[
\frac{1}{n}\left(\sum^{n}_{k=1}\theta^{\frac{k}{n}}\nu^{b}_{n,k}(a) \omega^{-k}\right)=\frac{1}{n}\left(\sum^{n}_{k=1}\nu^{b}_{n,k}(a) (\theta^{\frac{1}{n}}_{b}\omega^{-1})^{k}\right).
\]
The root $\omega$ occurs in $\kappa$ with the same multiplicity as $\theta^{-\frac{1}{n}}\omega$ appears in $\rho$.  Putting this all together we obtain the following proposition:

\begin{prop} \label{thm:main_theorem}
Let $\ca$ be a spherical fusion category. For any $a\in \ca$, and $b\in\text{Irr}(Z(\ca))$, the eigenvalues of $\rho^{b}_{n,a}$ are $n^{th}$ roots of $\theta^{-1}_{b}$. For such an $\omega$, its multiplicity as an eigenvalue is denoted $P^{b}_{n,a}(\omega^{-1})$ where 
\[P^{b}_{n,a}(x):=\sum^{n-1}_{k=0}\frac{\nu^{b}_{n,k}(a)}{n} x^{k}
\]
\end{prop}
\noindent In general, if the object $b\in Z(\ca)$ is {\it not} irreducible but is semi-simple, we can define 
$$K^{b}_{n,a}(\omega):=\sum_{c\in \text{Irr}(Z(\ca))} \delta(\omega^{n},\theta_{c}^{-1}) \dim(\Hom_{Z(\ca)}(c,b)) P^{c} _{n,a}(\omega^{-1}).$$
where $\delta$ is the Kronecker delta.

 
 \subsection{Computing GFS indicators: more for less}
 
It was shown by Ng-Schauenburg in \cite{NS3},\cite{NS2} that the GFS indicators for a spherical fusion category $\ca$ can be computed from the modular data of the Drinfel'd center $Z(\ca)$ and the $|\Irr(Z(\ca))|\times |\Irr(\ca)|$ matrix $A$ induced from the forgetful functor $\textrm{Forget}:Z(\ca) \to \ca$:
\[
A_{b,a}=\dim(\Hom_\ca({\rm Forget}(b), a))
\]

The full GFS data can be given as a matrix:
\begin{defi}
The {\bf GFS indicator matrix} $\mathcal{V}_{m,l}$ is the $|\text{Irr}(Z(\ca))|\times |\text{Irr}(\ca)|$ matrix defined by: 
\[
(\mathcal{V}_{m,l})_{b,a}=\nu^{b}_{m,l}(a)
\]
for $m>0$ and $l\in \Z$ with $\GCD(m,l)=1$.
\end{defi}
\noindent Let $\pi$ denote the $\SL_2(\mathbb{Z})$ representation given by the modular data of $Z(\ca)$. Let $(m,l)\cdot g$ denote the standard (right) action of $SL_{2}(\Z)$ on $\Z^{\oplus 2}$. Now we have the following useful facts:
\begin{theo}\cite{NS2}

\begin{enumerate}
    \item Let $\GCD(m,l)=1$, and $g\in SL_{2}(\Z)$ such that $(1,0)\cdot g^{-1}=(m,l)$.  Then $\mathcal{V}_{m,l}=\pi(g)A$. \cite[Theorem 5.4]{NS2} 
    \item $\pi$ factors through a finite group. \cite[Theorem 7.1]{NS2}
\end{enumerate}

\end{theo}
\noindent
This shows that if the modular data and linear structure of the forgetful functor are at hand then all GFS indicators can be computed. Furthermore, only finitely many numbers appear since $\pi$ factors through a finite group.  

The purpose of this section is to show that we can actually compute \textit{all} the GFS indicators from only the ``single rotation'' higher indicators $\nu^{b}_{n,1}(a)$.  Suppose $n\in \N$ and let $\zeta_{n}$ be a primitive $n^{th}$ root of unity. If $\GCD(k,n)=1$, let $\alpha_{k,n}$ denote the Galois automorphism on the cyclotomic field $\mathbb{Q}(\zeta_{n})$, which acts by $\zeta_{n}\to \zeta^{k}_{n}$.  Let $ a\in \ca$, and $b\in \text{Irr}(Z(\ca))$ be fixed, and  pick an $n^{th}$ root of $\theta_{b}$.

\begin{prop}\label{moreforless} For $1\le k\le n-1$
\[
\nu^{b}_{n,k}(a)=\theta^{-\frac{k}{n}}_{b}\alpha_{\frac{k}{g}, \frac{n}{g}}\left(\theta^{\frac{g}{n}}_{b} \nu^{b}_{\frac{n}{g},1}(a^{g}) \right) \quad \text{where $g = \GCD(k,n)$}
\]
\end{prop}
\begin{proof}
As in the previous section, define the operator $\kappa^{b}_{n,a}=\theta^{\frac{1}{n}}_{b} \rho^{b}_{n,a}$.  Then $\kappa^{b}_{n,a}$ is diagonalizable, and $(\kappa^{b}_{n,a})^{n}={\rm id}$, hence the eigenvalues of $\kappa$ are powers of $\zeta_{n}$, and ${\rm Tr}(\kappa^{l})\in \Z[\zeta]$.  It is clear that if $\GCD(k,n)=1$, then $\theta^{\frac{k}{n}}\nu^{b}_{n,k}(a)={\rm Tr}(\kappa^{k})=\alpha_{k, n}({\rm Tr}(\kappa))=\alpha_{k,n}(\theta^{\frac{1}{n}}\nu^{b}_{n,1}(a))$, hence the formula is true in this case.  It is then easy to see the general result since $\nu^{b}_{n,k}(a)=\nu^{b}_{\frac{n}{g},\frac{k}{g}}(a^{g})$ for $g=GCD(n,k)$.
\end{proof}

\begin{rem} Given a spherical fusion category $\ca$ one can compute the FS indicators from the matrix induced from the central induction functor as well as $T$ for the center. Gannon and Morrison have outlined an effective procedure for finding a finite set of possibilities for the central induction and $T$ matrix pairs for the center using the fusion rule of $\ca$ alone, which often (but not always) yields a unique such pair \cite{GM}.  Given an induction and $T$-matrix pair that does not actually correspond to a categorification of a given fusion rule, there seems to be no reason {\it a priori} why these numbers should live in the right number field suitable for applying the Galois automorphisms in the first place, or that the multiplicities of the eigenvalues should add up to the number as determined by the fusion rules. Understanding this would add an additional layer to the procedure of Gannon and Morrison, which could help to rule out extraneous possibilities of $S$ and $T$ pairs, or even to show that certain fusion rules admit no categorifications.
\end{rem}

\subsection{FS indicator rigidity}\label{sec:fs-rigid}

Let $R$ be a fusion ring.  Ocneanu rigidity says that there are only finitely many categorifications of $R$.  In general, one would like to classify all categorifications. 

\begin{defi}
A fusion ring $R$ exhibits \textbf{FS indicator rigidity} if two spherical categorifications of $R$ having the same FS indicator data implies that they are equivalent as fusion categories.  
\end{defi}

To be more precise, let $R$ be the fusion ring over $\mathbb{Z}$ with canonical basis $R_{B}=\{b_{i}\}^{n}_{i=1}$, and let $\mathcal{C}$ be a spherical categorification of $R$, meaning we have an identification of $\text{Irr}(\mathcal{C})$ with $R_{B}$ that induces a fusion ring homomorphism.  Let $M$ be the order of the $T$ matrix for $Z(\mathcal{C})$.  Recall that the (usual) FS indicators are given by $\nu_{j}(a)=\nu^{\mathbb{1}}_{j,1}(a)$.  Then $\ca$ has the $M\times n$ FS indicator matrix $\mathcal{V}_{\mathcal{C}}=(\nu_{i}(b_{j}))_{i,j}$.  This matrix captures all the information about the higher indicators since the sequence $\nu_{j}(a)$ is $M$-periodic by \cite[Theorem 5.5]{NS3}.  Hence the fusion ring $R$ has FS indicator rigidity if equality of the FS indicator matrices $\mathcal{V}_{\mathcal{C}}=\mathcal{V}_{\mathcal{D}}$ for two spherical categorifications $\mathcal{C}$ and $\mathcal{D}$ of $R$ implies that $\mathcal{C}\cong_{\otimes} \mathcal{D}$.

Several families of fusion rings have been shown to exhibit indicator rigidity. One well-known example are the Tambara-Yamagami fusion rings \cite{BJ}. Richard Ng has asked whether all fusion rings generated by a simple object exhibit indicator rigidity, but we present a negative answer here from quadratic categories. Recall the {\it Haagerup fusion ring}: 

\[
H = \mathbb{Z}[\mathbb{1}, g, g^{2}, \rho, \rho g, \rho g^{2}]
\] 
with mulitplication given by
\[
g^{3}=\mathbb{1}, \quad g\rho=\rho g^{2}, \quad \textrm{and} \quad \rho^{2}=\mathbb{1}+\rho+\rho g+\rho g^{2}
\]

\begin{prop} The fusion ring $H$ does not exhibit FS indicator rigidity.
\end{prop}

\begin{proof}
Izumi constructed two inequivalent unitary categorifications of this ring and found their Drinfel'd centers and induction matrices. Evans and Gannon found that the Drinfel'd centers of categories with this type of fusion ring (i.e. {\it Haargerup-Izumi fusion categories}) have modular data given by a {\it metric group}, that is a finite abelian group with a non-degenerate quadratic form. For the case of categories with ring $H$ the group is $\mathbb{Z}_{13}$; see \cite[\S 6.3]{EG15}. In, \cite{T}, the third author demonstrates that the FS indicator data is entirely determined by the quadratic form, and from \cite[Theorem 3(ii)]{EG15} we see that Izumi's two unitary categorifications correspond to the {\it same} quadratic form.  In fact, not only are the indicators the same, but the generalized indicators are the same: not only are the modular data for the centers of the two categorifications equivalent, but the $A$ matrices from their forgetful functors are also equal.
\end{proof}

This counterexample also demonstrates that even the stronger condition of asking for all \textit{generalized} FS indicator data to be the same is not sufficient to establish a similar rigidity property for general fusion rings even for GFS indicators.  

\subsection{Annular representation theory}

While indicator rigidity is not true in general, an interesting question remains in determining criteria on a fusion ring which would imply indicator rigidity. In this direction we have determined a weaker rigidity property satisfied by {\it unitary} (or $C^*$) fusion categories that categorify a {\it singly-generated} fusion ring. Restricting to the unitary setting allows us to utilize Jones's theory of {\it planar algebras}.

The celebrated Jones polynomial knot invariant can be obtained from skein theory as well as the representation theory of quantum groups \cite{ting}. These two methods rely on a common category called the {\bf Temperley-Lieb-Jones category (or algebroid)} and denoted $\mathcal{TLJ}(\delta)$. In the setting of quantum groups this is the tensor category of finite-dimensional representations of the quantum group $U_q(\mathfrak{sl}_2)$ where $\delta = q + q^{-1}$. In the setting of planar algebras $\mathcal{TLJ}(\delta)$ is obtained as the semisimple idempotent completion of the Temperley-Lieb algebroid, which has a familiar diagrammatic description due to Kauffman \cite{kauf}. Under this construction the simple objects of $\mathcal{TLJ}(\delta)$ are the well-known Jones-Wenzl idempotents. This realization is a refinement of a category of skein modules; see \cite{tur} for details.

For our purposes the simple objects of $\mathcal{TLJ}(\delta)$ are just natural numbers $n\in\mathbb{N}$, and morphism spaces $\Hom_{\mathcal{TLJ}}(n,m)$ are spanned by rectangular diagrams with $n$ bottom points and $m$ top points with non-crossing arcs pairing the points, i.e. Temperley-Lieb diagrams on $n+m$ points. Composition of morphisms is by stacking diagrams, and closed circles evaluate to the number $\delta$.  The category $\mathcal{TLJ}(\delta)$ is semisimple unless $\delta\in \{2\cos{\frac{\pi}{n}}: n\ge 3\}$.  From here on we will assume $\delta$ is a semisimple value. This category is universal in the sense that given any unitary rigid tensor category $\ca$ generated by a symmetrically self-dual object there exists a unique dominant dimension-preserving monoidal functor $\mathcal{TLJ}(\delta) \to \ca$.



Note that $\mathcal{TLJ}(\delta)$ is {\it not} a fusion category since there are infinitely many simple objects, hence we lose much of the good behavior of the Drinfel'd center. There have been several approaches to understanding half-braidings. For this purpose Jones introduced the {\bf annular Temperley-Lieb category (or algebroid)} $\mathcal{A}nn\mathcal{TL}(\delta)$. Its irreducible unitary modules yield half-braidings for objects in $\mathcal{TLJ}(\delta)$.

In the annular Temperley-Lieb algebroid objects are again given by natural numbers and morphisms are given by the span of \textit{annular} Temperley-Lieb diagrams. These are diagrams on an annulus with $n$ points on the inner boundary and $m$ points on the outer boundary. (For formal definitions see \cite{JR} and \cite{Jonann}, and note that this object is called the {\it affine} Temperley-Lieb algebroid in \cite{Jonann} with a likewise change to the notation.)  This category is quite far away from being semi-simple and it is not even a tensor category.  Its representation category, while still being non-semi-simple, is a braided tensor category; see \cite{DGG}.

A more purely categorical approach to understanding this representation category is provided by Neshveyev and Yamashita in \cite{NY}. They show that the representation category for the annular Temperley-Lieb algebroid is equivalent to the Drinfel'd center of the {\bf ind-category}:

\begin{defi}\cite[\S 2.2]{NY}
Let $\ca$ be a semisimple unitary tensor category. The category ${\rm ind}-\ca$ is given as follows:
\begin{itemize}
\item {\it objects:} inductive systems in $\ca$
\[
x_* = \{u_{ji}\in \Hom_\ca(x_i, x_j)\}_{i<j} 
\]
where the $u_{ji}$ are isometries.
\item {\it morphisms $t: x_* \to y_*$}: collections of morphisms in $\ca$
\[
t = \{ t_{ki}\in \Hom_\ca(x_i,y_k) \}
\]
such that
\[
t_{kj}u_{ji} = t_{ki} \quad {\rm if} \; i<j
\]
\[
v^*_{lk}t_{li}=t_{ki} \quad {\rm if} \; k<l
\]
\[
\sup_{k,i} ||t_{ki}|| < \infty
\]
\end{itemize}
An object $v$ of ${\rm ind} - \ca$ is called \textbf{locally finite} if $v\in \text{ind-}\ca$ satisfies $\text{dim}(\Hom_{\text{ind}-\ca}(a, v))<\infty$ for all $a\in \ca$. 
\end{defi}

Jones's annular representation category is equivalent to $Z({\rm ind} - \mathcal{TLJ}(\delta))$ as braided tensor categories. Using the annular approach Jones and Jones-Reznikoff establish a classification of unitary locally-finite irreducible representations for the representation category of the annular (also called ``affine'' in \cite{JR}) Temperley-Lieb algebroid. This provides us the following: 
 
  
\begin{prop}\cite{Jonann, JR}
An irreducible locally-finite unitary object 
\[
v_* \in Z({\rm ind} - \mathcal{TLJ}(\delta))
\]
is classified up to isomorphism by the following data:
\begin{enumerate}
\item
A number $n\ge 0$ called the \textit{lowest weight} of $v_*$ such that $v_{k}=0$ for $k<n$ and $v_{n}\ne 0$
\item
For $n>0$: a scalar $\omega\in S^{1}$ (which is the rotation eigenvalue)\\ For $n=0$: a number $0\le \mu \le \delta^{2}$ (which is the eigenvalue of the double closed loop)
\end{enumerate}
\end{prop}

Every locally-finite unitary object of $Z({\rm ind}-\ca)$ is semi-simple \cite{Jonann}, i.e. every such object is isomorphic to a (possibly infinite) direct sum of simple obects. (Note that at each weight space only finitely many such simple objects can appear by local finiteness.) By the proposition above, each simple $v_{n}$ for $n>0$ is classified up to isomorphism by the set of eigenvalues of the rotation operator.  The only exception is the weight $0$ part, where there is no rotation.  This gives us the following corollary:


\begin{cor}\label{cor:eigenclass}
Locally finite unitary objects $v_*\in Z({\rm ind}-\mathcal{TLJ}(\delta))$ with a fixed weight $0$ component $v_{0}$, are classified up to isomorphism by the set of eigenvalues $\Omega_{n}$ of the rotation operator on each $v_{n}$ for $n>0$.
\end{cor}

\subsection{Annular FS indicator rigidity}

Now, again let $\ca$ be a spherical fusion category. We now give the definition of Jones's planar algebras in categorical terms:
\begin{defi}
Let $x\in \ca$ such that $d_x\geq 2$ and define:
\[
P^x_n := \Hom_\ca (\mathbb{1},(x\otimes \overline{x})^{n})
\]
Then the spherical structure makes the system $P^x_*$ a {\it commutative algebra object} in $Z({\rm ind} - \mathcal{TLJ}(\delta))$ called the {\bf planar algebra} for $x\in \ca$. (Note that $\delta=d_x$; the algebra structure is given by the map $ev_x$ in $\ca$.)
\end{defi}

Now, suppose $R$ is a fusion ring and $\ca$ and $\mathcal{D}$ are categorifications of $R$. Fix specified isomorphisms of  $\ca$ and $\mathcal{D}$ to $ R$.  We say $a\sim b$ if $[a]= [b]\in R$ for $a\in \ca$ and $b\in \mathcal{D}$.  Note that in particular, $a\sim b$ implies they have the same dimension.

\begin{theo}\label{annFSrigid} \emph{{\bf (Annular FS indicator rigidity)}} Suppose $R$ is a fusion ring with two spherical categorifications $\ca$ and $\mathcal{D}$ with equivalent FS indicator data, i.e. such that $\mathcal{V}_{\mathcal{C}}=\mathcal{V}_{\mathcal{D}}$.  If $a\in \ca$ and $b\in \mathcal{D}$ with $a\sim b$, and $d_{a}\ge 2$, then $P^{a}_*$ and $P^{b}_*$ are isomorphic as \emph{\underline{objects}} in $Z({\rm ind} - \mathcal{TLJ}(\delta)) \cong Z({\rm ind} -Rep(U_q(\mathfrak{sl}_2)))$.
\end{theo}

\begin{proof}
It is clear that the weight $0$ components of $P^{a}_*$ and $P^{b}_*$ are both given by the number $d^{2}_{a}=d^{2}_{b}$, hence have isomorphic weight $0$ components,  By Corollary \ref{cor:eigenclass}, it suffices to show that the eigenvalues of the rotation are the same as sets for all non-zero weights.  Since the FS indicator data is the same, combining Theorem \ref{thm:main_theorem} and Theorem \ref{moreforless} yields the desired result.
\end{proof}

\section{Eigenvalues for braid group representations from ribbon categories}

\subsection{Braid group actions and their spectra}

Let $\da$ be a ribbon fusion category. The braid group $\mathcal{B}_n$ acts on the vector space $\Hom_\ca(b,a^n)$ via the braiding in the category. We call the corresponding representation
\[
\pi_{a,n}^b: \mathcal{B}_n \to {\rm GL}(\Hom_\da(b,a^n))
\]

Given an object $a\in \da$ the tensor structure on $\da$ gives an inclusion $\End_\da(a^n) \to \End_\da(a^{n+1})$ which is compatible with the natural inclusion $\mathcal{B}_n \to \mathcal{B}_{n+1}$. The braiding gives the following homomorphism:
\[
\pi_{a,n} : \mathcal{B}_n \to \End_\da(a^n)
\]
 Taking colimits gives the homomorphism: \[
 \pi_a : \mathcal{B}_{\infty} \to \cup_n \End_\da(a^n)
 \]
 This implies that the spectrum of $\pi_{a,n}(g)$ only depends on the conjugacy class of $g$ inside $\mathcal{B}_{\infty}$.  (In other words, the representations of our braid group are ``local.'')  Furthermore, suppose a braid $x$ is conjugate in $\mathcal{B}_{\infty}$ to some braid $y\in \mathcal{B}_{n}$.  Then to determine the spectrum of $\pi_{a}(x)$ it suffices to determine the spectrum of $\pi^{b}_{a,n}(y)$ acting on $\Hom_\da(b, a^{n})$ for all $b\in \text{Irr}(\da)$.  In other words:
 \[
 \Spec(\pi_{a}(x))=\bigcup_{b\in \text{Irr}(\da)} \Spec(\pi^{b}_{a,n}(y)).
 \]
 
\subsection{Diagonalizing} 

Consider for $l,m\ge 0$, $l+m<n$ the {\bf Jucys-Murphy elements} of the braid group $B_{n}$ given by

\[A^{n}_{l,m}:=\begin{tikzpicture}[yscale=-1,scale=0.0333,baseline={([yshift=-.5ex]current bounding box.center)}]
\begin{scope}[shift={(0.00mm,719.29mm)}]
\draw [fill=none,draw=black] (157.14mm,-488.75mm)
-- ++(5.72mm,930.22mm)
;
\draw [fill=none,draw=black] (162.85mm,568.51mm)
-- ++(0.00mm,277.77mm)
;
\draw [fill=none,draw=black] (357.14mm,-488.75mm)
-- ++(5.72mm,826.86mm)
;
\draw [fill=none,draw=black] (362.86mm,480.23mm)
-- ++(0.00mm,366.06mm)
;
\draw [fill=none,draw=black] (557.14mm,-488.75mm)
-- ++(5.72mm,747.18mm)
;
\draw [fill=none,draw=black] (562.86mm,407.02mm)
-- ++(0.00mm,439.27mm)
;
\draw [fill=none,draw=black] (787.14mm,-488.75mm)
-- ++(5.73mm,656.74mm)
;
\draw [fill=none,draw=black] (792.86mm,293.33mm)
-- ++(0.00mm,552.96mm)
;
\draw [fill=none,draw=black] (1043.14mm,-488.74mm)
-- ++(5.73mm,589.99mm)
;
\draw [fill=none,draw=black] (1048.86mm,216.33mm)
-- ++(0.00mm,629.96mm)
;
\draw [fill=none,draw=black] (1303.14mm,-488.74mm)
-- ++(5.73mm,484.48mm)
;
\draw [fill=none,draw=black] (1308.86mm,118.76mm)
-- ++(0.00mm,727.53mm)
;
\draw [fill=none,draw=black] (1541.14mm,-488.74mm)
-- ++(5.73mm,410.68mm)
;
\draw [fill=none,draw=black] (1546.86mm,34.56mm)
-- ++(0.00mm,811.73mm)
;
\draw [fill=none,draw=black] (471.43mm,839.83mm)
.. controls (483.55mm,779.55mm) and (446.57mm,745.83mm) .. (393.57mm,739.70mm)
;
\draw [fill=none,draw=black] (329.28mm,732.16mm)
.. controls (278.42mm,722.78mm) and (203.99mm,699.52mm) .. (178.57mm,694.48mm)
;
\draw [fill=none,draw=black] (122.85mm,682.64mm)
.. controls (-394.46mm,551.75mm) and (1530.12mm,-18.55mm) .. (1645.72mm,-55.94mm)
;
\draw [fill=none,draw=black] (1645.30mm,-56.11mm)
.. controls ++(131.54mm,-40.41mm) and ++(165.50mm,25.43mm) .. ++(-69.98mm,-157.06mm)
;
\draw [fill=none,draw=black] (1512.86mm,-222.82mm)
.. controls ++(-63.20mm,-12.67mm) and ++(53.35mm,19.52mm) .. ++(-181.43mm,-60.29mm)
;
\draw [fill=none,draw=black] (1281.43mm,-294.95mm)
.. controls ++(-111.91mm,-42.90mm) and ++(22.01mm,70.81mm) .. ++(-132.86mm,-191.64mm)
;
\draw [fill=none,draw=black] (1246.38mm,-501.23mm)
.. controls ++(19.59mm,-88.32mm) and ++(-9.57mm,103.14mm) .. ++(170.51mm,-101.06mm)
;
\draw [fill=none,draw=black] (1587.16mm,-501.23mm)
.. controls ++(-19.59mm,-88.32mm) and ++(9.57mm,103.14mm) .. ++(-170.51mm,-101.06mm)
;
\draw [fill=none,draw=black] (91.79mm,859.80mm)
.. controls (110.40mm,948.13mm) and (244.65mm,857.72mm) .. (253.74mm,960.87mm)
;
\draw [fill=none,draw=black] (415.46mm,859.80mm)
.. controls (396.86mm,948.13mm) and (262.60mm,857.72mm) .. (253.51mm,960.87mm)
;
\node [black] at (1400.14mm,-653.35mm) { $m$ };
\node [black] at (245.71mm,1029.51mm) { $l$ };
\end{scope}
\end{tikzpicture}.\]
Note that we could define $A^{n,-}_{l,m}$ similarly with over and under crossings switched, and the rest of this section follows from similar arguments. Letting $\sigma_{i}$ be the standard braid group generators, inside $B_{\infty}$, $\sigma_{i}$ is conjugate to $A_{0,0}^{2}$ and  $\sigma_{i+1} \sigma_{i}\sigma_{i+1}=\sigma_{i}\sigma_{i+1}\sigma_{i}$ is conjugate to $A^{3}_{1,0}$. 

Using the results from the last section we can compute the eigenvalues for $A^{n}_{l,m}$ under the above representations in terms of generalized Frobenius-Schur indicators. Let $\ca$ be ribbon fusion, and let $a,b\in \ca$.  We define the linear isomorphisms 
\[
\ V^{n}_{l,m}: \Hom_\ca(b, a^{\otimes n})\rightarrow \Hom_\ca( \overline{a}^{ l+ m} \otimes b, a^{n-(l+m)})
\]
by the following formulae:
\[
V^{n}_{l,m}(f):=({\rm ev}_{\overline{a}^{ l}}\otimes 1_{a^{ n-(m+l)}}\otimes {\rm ev}_{\overline{a}^{ m}})\circ( 1_{\overline{a}^{ l}}\otimes f\otimes 1_{\overline{a}^{ m}})\circ (1_{\overline{a}^{ l}}\otimes c_{\overline{a}^{ m},b}).
\]
\[
V^{n}_{l,m}:\ \begin{tikzpicture}[scale=0.05,yscale=-1,
baseline={([yshift=-.5ex]current bounding box.center)}]
\draw [fill=none,draw=black] (191.43mm,372.36mm)
-- ++(0.00mm,125.71mm)
-- ++(331.43mm,0.00mm)
-- ++(0.00mm,-127.14mm)
-- cycle
;
\draw [fill=none,draw=black] (230.00mm,371.51mm)
-- ++(0.00mm,-141.43mm)
-- ++(0.00mm,0.00mm)
;
\draw [fill=none,draw=black] (274.29mm,372.36mm)
.. controls ++(0.00mm,-142.86mm) and ++(0.00mm,0.00mm) .. ++(0.00mm,-142.86mm)
;
\draw [fill=none,draw=black] (445.71mm,371.51mm)
.. controls ++(0.00mm,-141.43mm) and ++(0.00mm,0.00mm) .. ++(0.00mm,-141.43mm)
;
\draw [fill=black,draw=black] (307.14mm,290.93mm) circle (0.92mm) ;
\draw [fill=black,draw=black] (327.14mm,290.93mm) circle (0.92mm) ;
\draw [fill=black,draw=black] (347.14mm,290.93mm) circle (0.92mm) ;
\draw [fill=black,draw=black] (367.14mm,290.93mm) circle (0.92mm) ;
\draw [fill=black,draw=black] (387.14mm,290.93mm) circle (0.92mm) ;
\draw [fill=black,draw=black] (407.14mm,290.93mm) circle (0.92mm) ;
\draw [fill=none,draw=black] (348.56mm,498.08mm)
.. controls ++(4.30mm,144.27mm) and ++(0.00mm,0.00mm) .. ++(4.30mm,144.27mm)
;
\draw [fill=none,draw=black] (217.69mm,291.22mm)
.. controls ++(7.15mm,-0.48mm) and ++(-1.47mm,1.81mm) .. ++(11.50mm,-14.16mm)
;
\draw [fill=none,draw=black] (238.90mm,291.21mm)
.. controls ++(-5.75mm,-2.70mm) and ++(1.18mm,1.66mm) .. ++(-10.10mm,-14.14mm)
;
\draw [fill=none,draw=black] (263.15mm,291.22mm)
.. controls ++(7.15mm,-0.48mm) and ++(-1.47mm,1.81mm) .. ++(11.50mm,-14.16mm)
;
\draw [fill=none,draw=black] (284.36mm,291.21mm)
.. controls ++(-5.75mm,-2.70mm) and ++(1.18mm,1.66mm) .. ++(-10.10mm,-14.14mm)
;
\draw [fill=none,draw=black] (434.88mm,296.27mm)
.. controls ++(7.15mm,-0.48mm) and ++(-1.47mm,1.81mm) .. ++(11.50mm,-14.16mm)
;
\draw [fill=none,draw=black] (456.08mm,296.26mm)
.. controls ++(-5.75mm,-2.70mm) and ++(1.18mm,1.66mm) .. ++(-10.10mm,-14.14mm)
;
\draw [fill=none,draw=black] (338.91mm,572.04mm)
.. controls ++(7.15mm,-0.48mm) and ++(-1.47mm,1.81mm) .. ++(11.50mm,-14.16mm)
;
\draw [fill=none,draw=black] (360.12mm,572.03mm)
.. controls ++(-5.75mm,-2.70mm) and ++(1.18mm,1.66mm) .. ++(-10.10mm,-14.14mm)
;
\draw [fill=none,draw=black] (336.21mm,180.89mm)
.. controls ++(-26.39mm,48.12mm) and ++(36.84mm,-32.57mm) .. ++(-119.07mm,39.32mm)
;
\draw [fill=none,draw=black] (336.16mm,180.93mm)
.. controls ++(25.89mm,47.21mm) and ++(-36.14mm,-31.96mm) .. ++(116.81mm,38.58mm)
;
\node [black] at (355.71mm,130.22mm) { $a^n$ };
\node [black] at (352.86mm,439.51mm) { $f$ };
\node [black] at (350.57mm,690.93mm) { $b$ };
\end{tikzpicture}
\quad \mapsto \quad
\begin{tikzpicture}[yscale=-1,scale=0.07,
baseline={([yshift=-.5ex]current bounding box.center)}]
\draw [fill=none,draw=black] (200.00mm,456.65mm)
.. controls ++(0.00mm,78.57mm) and ++(0.00mm,0.00mm) .. ++(0.00mm,78.57mm)
-- ++(282.86mm,0.00mm)
-- ++(0.00mm,-75.71mm)
-- cycle
;
\draw [fill=none,draw=black] (304.29mm,210.93mm)
.. controls ++(6.58mm,62.11mm) and ++(-3.10mm,-49.65mm) .. ++(4.29mm,247.14mm)
;
\draw [fill=none,draw=black] (357.14mm,215.22mm)
.. controls ++(12.81mm,63.14mm) and ++(-3.70mm,-55.57mm) .. ++(1.43mm,242.86mm)
;
\draw [fill=none,draw=black] (237.14mm,456.83mm)
.. controls (228.14mm,427.09mm) and (206.28mm,352.08mm) .. (160.00mm,375.22mm)
.. controls ++(-85.32mm,42.66mm) and ++(0.00mm,-78.84mm) .. ++(-47.14mm,370.00mm)
;
\draw [fill=none,draw=black] (274.46mm,457.01mm)
.. controls ++(-2.98mm,-39.13mm) and ++(20.93mm,34.88mm) .. ++(-27.14mm,-81.43mm)
.. controls ++(-19.92mm,-33.20mm) and ++(35.43mm,7.09mm) .. ++(-93.04mm,-66.07mm)
.. controls (32.36mm,285.12mm) and (52.59mm,677.36mm) .. (61.43mm,748.08mm)
;
\draw [fill=none,draw=black] (455.71mm,459.26mm)
.. controls ++(11.92mm,-26.31mm) and ++(-28.09mm,0.00mm) .. ++(68.57mm,-45.47mm)
.. controls ++(64.19mm,0.00mm) and ++(47.05mm,-47.05mm) .. ++(70.00mm,131.43mm)
.. controls ++(-117.36mm,117.36mm) and ++(0.00mm,-196.03mm) .. ++(-331.43mm,201.43mm)
;
\draw [fill=none,draw=black] (418.57mm,459.51mm)
.. controls ++(13.97mm,-48.49mm) and ++(-44.53mm,17.81mm) .. ++(94.29mm,-95.71mm)
.. controls ++(103.94mm,-41.58mm) and ++(64.98mm,-81.23mm) .. ++(131.43mm,202.86mm)
.. controls ++(-85.31mm,106.63mm) and ++(0.00mm,-102.67mm) .. ++(-307.14mm,181.43mm)
;
\draw [fill=none,draw=black] (347.79mm,534.40mm)
.. controls ++(-2.71mm,35.47mm) and ++(-16.57mm,-19.89mm) .. ++(36.07mm,70.47mm)
;
\draw [fill=none,draw=black] (402.56mm,627.04mm)
.. controls ++(10.20mm,12.91mm) and ++(-8.64mm,-10.37mm) .. ++(21.68mm,29.35mm)
;
\draw [fill=none,draw=black] (438.41mm,672.52mm)
.. controls ++(20.45mm,25.46mm) and ++(0.00mm,-29.23mm) .. ++(9.09mm,79.80mm)
;
\draw [fill=none,draw=black] (277.07mm,185.69mm)
.. controls ++(4.87mm,-27.66mm) and ++(-10.59mm,29.34mm) .. ++(54.19mm,-23.43mm)
;
\draw [fill=none,draw=black] (385.29mm,185.59mm)
.. controls ++(-4.87mm,-27.66mm) and ++(10.59mm,29.34mm) .. ++(-54.19mm,-23.43mm)
;
\draw [fill=none,draw=black] (36.67mm,758.23mm)
.. controls ++(4.87mm,27.66mm) and ++(-10.59mm,-29.34mm) .. ++(54.19mm,23.43mm)
;
\draw [fill=none,draw=black] (144.89mm,758.33mm)
.. controls ++(-4.87mm,27.66mm) and ++(10.59mm,-29.34mm) .. ++(-54.19mm,23.43mm)
;
\draw [fill=none,draw=black] (242.74mm,760.25mm)
.. controls ++(4.87mm,27.66mm) and ++(-10.59mm,-29.34mm) .. ++(54.19mm,23.43mm)
;
\draw [fill=none,draw=black] (350.96mm,760.35mm)
.. controls ++(-4.87mm,27.66mm) and ++(10.59mm,-29.34mm) .. ++(-54.19mm,23.43mm)
;
\node [black] at (310.57mm,115.22mm) { $n - (m+l)$ };
\node [black] at (81.43mm,830.51mm) { $l$ };
\node [black] at (292.74mm,830.23mm) { $m$ };
\node [black] at (327.29mm,495.00mm) { $f$ };
\draw [fill=none,draw=black] (297.14mm,309.28mm)
.. controls ++(6.07mm,0.17mm) and ++(-1.09mm,6.25mm) .. ++(10.32mm,-14.48mm)
;
\draw [fill=none,draw=black] (317.71mm,309.22mm)
.. controls ++(-6.07mm,0.17mm) and ++(1.09mm,6.25mm) .. ++(-10.32mm,-14.48mm)
;
\draw [fill=none,draw=black] (351.43mm,307.00mm)
.. controls ++(6.07mm,0.17mm) and ++(-1.09mm,6.25mm) .. ++(10.32mm,-14.48mm)
;
\draw [fill=none,draw=black] (371.99mm,306.93mm)
.. controls ++(-6.07mm,0.17mm) and ++(1.09mm,6.25mm) .. ++(-10.32mm,-14.48mm)
;
\draw [fill=none,draw=black] (45.00mm,569.85mm)
.. controls ++(6.07mm,-0.17mm) and ++(-1.09mm,-6.25mm) .. ++(10.32mm,14.48mm)
;
\draw [fill=none,draw=black] (65.56mm,569.92mm)
.. controls ++(-6.07mm,-0.17mm) and ++(1.09mm,-6.25mm) .. ++(-10.32mm,14.48mm)
;
\draw [fill=none,draw=black] (94.29mm,569.14mm)
.. controls ++(6.07mm,-0.17mm) and ++(-1.09mm,-6.25mm) .. ++(10.32mm,14.48mm)
;
\draw [fill=none,draw=black] (114.85mm,569.21mm)
.. controls ++(-6.07mm,-0.17mm) and ++(1.09mm,-6.25mm) .. ++(-10.32mm,14.48mm)
;
\draw [fill=none,draw=black] (269.02mm,672.56mm)
.. controls ++(5.34mm,2.88mm) and ++(2.17mm,-5.96mm) .. ++(1.70mm,17.70mm)
;
\draw [fill=none,draw=black] (286.79mm,682.89mm)
.. controls ++(-5.17mm,-3.18mm) and ++(4.07mm,-4.86mm) .. ++(-16.18mm,7.39mm)
;
\draw [fill=none,draw=black] (371.07mm,681.03mm)
.. controls ++(3.40mm,5.03mm) and ++(4.68mm,-4.27mm) .. ++(-6.68mm,16.48mm)
;
\draw [fill=none,draw=black] (382.04mm,698.42mm)
.. controls ++(-3.11mm,-5.21mm) and ++(5.86mm,-2.43mm) .. ++(-17.76mm,-0.94mm)
;
\node [black] at (485.71mm,743.79mm) { $b$ };
\node [black] at (122.86mm,273.79mm) { $a$ };
\node [black] at (570.00mm,329.51mm) { $a$ };
\end{tikzpicture}
\]
Let $G : \ca \boxtimes \tilde{\ca} \to Z(\ca)$ be the usual inclusion. Then
\[
G(\overline{a}^{l+m}\boxtimes  \tilde{b})=(\overline{a}^{l+m}\otimes b, (c_{\overline{a}^{m+l}, \cdot}\otimes 1_{b})\circ (1_{\overline{a}^{l}}\otimes  \tilde{c}_{b, \cdot}) ).
\]
If we identify $\ca \boxtimes \tilde{\ca}$ with its image under $G$, then we have
\[
(V^{n}_ {l,m})^{-1}\ \rho^{ \overline{a}^{l+m}\boxtimes \tilde{b}}_{n-(m+l), a}\ V^{n}_{l,m}=\theta_{a}\pi^{b}_{a}(A^{ n}_{l,m}).
\]
The picture is given by 
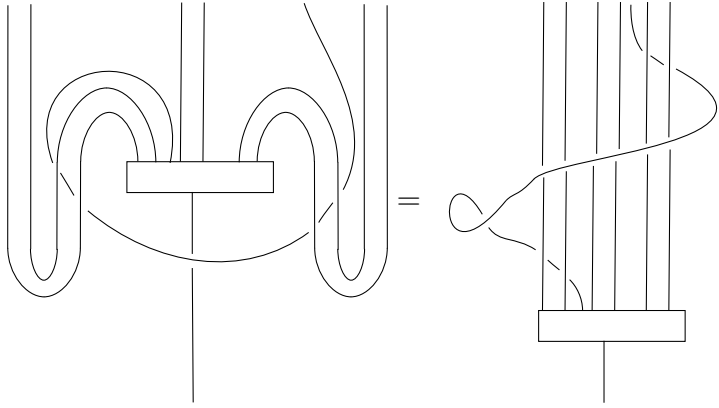
\[
\begin{tikzpicture}[yscale=-1,scale=0.03,baseline={([yshift=-.5ex]current bounding box.center)}]
\begin{scope}[shift={(0.00mm,719.29mm)}]
\draw [fill=none,draw=black] (548.65mm,-13.20mm)
-- ++(0.00mm,136.83mm)
-- ++(649.31mm,0.00mm)
-- ++(0.00mm,-136.83mm)
-- cycle
;
\draw [fill=none,draw=black] (791.43mm,-719.07mm)
-- ++(-5.71mm,705.71mm)
;
\draw [fill=none,draw=black] (891.43mm,-719.07mm)
-- ++(-5.71mm,705.71mm)
;
\draw [fill=none,draw=black] (597.24mm,-14.75mm)
.. controls (597.02mm,-129.09mm) and (533.49mm,-232.09mm) .. (470.03mm,-231.96mm)
;
\draw [fill=none,draw=black] (344.03mm,-14.75mm)
.. controls (344.25mm,-129.09mm) and (407.79mm,-232.09mm) .. (471.24mm,-231.96mm)
;
\draw [fill=none,draw=black] (343.45mm,-16.38mm)
-- (342.44mm,370.51mm)
;
\draw [fill=none,draw=black] (342.35mm,369.16mm)
.. controls (342.07mm,483.47mm) and (261.37mm,586.43mm) .. (180.76mm,586.29mm)
;
\draw [fill=none,draw=black] (20.72mm,369.16mm)
.. controls (20.99mm,483.47mm) and (101.70mm,586.43mm) .. (182.30mm,586.29mm)
;
\draw [fill=none,draw=black] (21.44mm,-706.87mm)
-- (20.45mm,372.13mm)
;
\draw [fill=none,draw=black] (676.94mm,-10.41mm)
.. controls (676.57mm,-184.07mm) and (566.99mm,-340.49mm) .. (457.54mm,-340.28mm)
;
\draw [fill=none,draw=black] (240.23mm,-10.41mm)
.. controls (240.60mm,-184.07mm) and (350.18mm,-340.49mm) .. (459.63mm,-340.28mm)
;
\draw [fill=none,draw=black] (239.45mm,-12.03mm)
-- (238.44mm,377.57mm)
;
\draw [fill=none,draw=black] (240.06mm,373.51mm)
.. controls ++(-0.10mm,73.13mm) and ++(29.75mm,0.09mm) .. ++(-59.64mm,138.91mm)
;
\draw [fill=none,draw=black] (121.35mm,373.51mm)
.. controls ++(0.10mm,73.13mm) and ++(-29.75mm,0.09mm) .. ++(59.64mm,138.91mm)
;
\draw [fill=none,draw=black] (123.23mm,-711.60mm)
-- (122.39mm,375.01mm)
;
\draw [fill=none,draw=black] (838.43mm,121.00mm)
-- ++(0.00mm,274.76mm)
-- ++(0.00mm,0.00mm)
;
\draw [fill=none,draw=black] (838.43mm,458.39mm)
-- ++(4.04mm,595.99mm)
-- ++(0.00mm,0.00mm)
;
\draw [fill=none,draw=black] (1126.03mm,-14.75mm)
.. controls ++(0.22mm,-114.34mm) and ++(-63.46mm,-0.13mm) .. ++(127.21mm,-217.21mm)
;
\draw [fill=none,draw=black] (1379.23mm,-14.75mm)
.. controls ++(-0.22mm,-114.34mm) and ++(63.46mm,-0.13mm) .. ++(-127.21mm,-217.21mm)
;
\draw [fill=none,draw=black] (1379.82mm,-16.38mm)
-- ++(1.01mm,386.89mm)
;
\draw [fill=none,draw=black] (1380.92mm,369.16mm)
.. controls ++(0.27mm,114.30mm) and ++(-80.61mm,0.13mm) .. ++(161.58mm,217.13mm)
;
\draw [fill=none,draw=black] (1702.55mm,369.16mm)
.. controls ++(-0.27mm,114.30mm) and ++(80.61mm,0.13mm) .. ++(-161.58mm,217.13mm)
;
\draw [fill=none,draw=black] (1701.82mm,-706.87mm)
-- ++(0.99mm,1079.00mm)
;
\draw [fill=none,draw=black] (1046.32mm,-10.41mm)
.. controls ++(0.37mm,-173.65mm) and ++(-109.45mm,-0.20mm) .. ++(219.41mm,-329.87mm)
;
\draw [fill=none,draw=black] (1483.04mm,-10.41mm)
.. controls (1482.67mm,-184.07mm) and (1373.08mm,-340.49mm) .. (1263.64mm,-340.28mm)
;
\draw [fill=none,draw=black] (1483.82mm,-12.03mm)
-- ++(1.01mm,389.60mm)
;
\draw [fill=none,draw=black] (1483.21mm,373.51mm)
.. controls ++(0.10mm,73.13mm) and ++(-29.75mm,0.09mm) .. ++(59.64mm,138.91mm)
;
\draw [fill=none,draw=black] (1601.92mm,373.51mm)
.. controls ++(-0.10mm,73.13mm) and ++(29.75mm,0.09mm) .. ++(-59.64mm,138.91mm)
;
\draw [fill=none,draw=black] (1600.04mm,-711.60mm)
-- ++(0.84mm,1086.60mm)
;
\draw [fill=none,draw=black] (740.00mm,-7.64mm)
.. controls (855.77mm,-575.95mm) and (37.05mm,-522.73mm) .. (220.00mm,-7.64mm)
;
\draw [fill=none,draw=black] (254.29mm,36.65mm)
.. controls ++(19.50mm,30.29mm) and ++(-15.48mm,-25.00mm) .. ++(60.00mm,100.00mm)
;
\draw [fill=none,draw=black] (375.71mm,204.65mm)
.. controls (679.84mm,457.06mm) and (1070.31mm,510.93mm) .. (1354.29mm,300.36mm)
;
\draw [fill=none,draw=black] (1398.57mm,256.65mm)
.. controls ++(16.00mm,-28.72mm) and ++(-10.17mm,14.24mm) .. ++(62.86mm,-88.01mm)
;
\draw [fill=none,draw=black] (1507.14mm,110.93mm)
.. controls ++(159.03mm,-303.52mm) and ++(71.36mm,214.09mm) .. ++(-172.86mm,-827.14mm)
;
\end{scope}
\end{tikzpicture}=
\begin{tikzpicture}[yscale=-1,scale=0.03,baseline={([yshift=-.5ex]current bounding box.center)}]
\begin{scope}[shift={(0.00mm,719.29mm)}]
\draw [fill=none,draw=black] (548.65mm,646.80mm)
-- ++(0.00mm,136.83mm)
-- ++(649.31mm,0.00mm)
-- ++(0.00mm,-136.83mm)
-- cycle
;
\draw [fill=none,draw=black] (791.44mm,5.46mm)
-- (785.73mm,646.66mm)
;
\draw [fill=none,draw=black] (891.41mm,-15.25mm)
-- (885.70mm,646.66mm)
;
\draw [fill=none,draw=black] (838.43mm,781.00mm)
-- ++(0.00mm,274.76mm)
-- ++(0.00mm,0.00mm)
;
\draw [fill=none,draw=black] (1031.45mm,-45.56mm)
-- (1025.75mm,646.67mm)
;
\draw [fill=none,draw=black] (1131.42mm,-67.27mm)
-- ++(-5.70mm,713.94mm)
;
\draw [fill=none,draw=black] (571.43mm,56.47mm)
-- ++(-5.71mm,590.19mm)
;
\draw [fill=none,draw=black] (671.40mm,29.20mm)
-- (665.69mm,646.66mm)
;
\draw [fill=none,draw=black] (743.47mm,645.27mm)
.. controls ++(-0.47mm,-53.83mm) and ++(31.06mm,29.33mm) .. ++(-57.58mm,-135.36mm)
;
\draw [fill=none,draw=black] (640.96mm,466.92mm)
.. controls ++(-21.11mm,-16.80mm) and ++(14.22mm,12.64mm) .. ++(-52.04mm,-43.88mm)
;
\draw [fill=none,draw=black] (535.38mm,377.58mm)
.. controls (441.87mm,316.40mm) and (377.59mm,349.23mm) .. (327.29mm,278.59mm)
;
\draw [fill=none,draw=black] (299.01mm,220.00mm)
.. controls (109.21mm,-101.59mm) and (67.54mm,569.79mm) .. (404.06mm,159.39mm)
;
\draw [fill=none,draw=black] (528.23mm,63.46mm)
.. controls (634.48mm,-42.50mm) and (1768.43mm,-95.90mm) .. (1157.50mm,-424.60mm)
;
\draw [fill=none,draw=black] (905.75mm,-69.30mm)
-- ++(5.71mm,-652.09mm)
;
\draw [fill=none,draw=black] (805.77mm,-48.24mm)
-- ++(5.70mm,-673.15mm)
;
\draw [fill=none,draw=black] (665.73mm,-17.42mm)
-- ++(5.70mm,-703.97mm)
;
\draw [fill=none,draw=black] (565.76mm,4.67mm)
-- ++(5.70mm,-726.06mm)
;
\draw [fill=none,draw=black] (1125.76mm,-121.18mm)
-- ++(5.71mm,-600.21mm)
;
\draw [fill=none,draw=black] (1025.79mm,-93.44mm)
-- ++(5.71mm,-627.95mm)
;
\draw [fill=none,draw=black] (1103.09mm,-442.66mm)
.. controls ++(-19.79mm,-14.40mm) and ++(6.08mm,4.25mm) .. ++(-58.59mm,-40.41mm)
;
\draw [fill=none,draw=black] (1008.13mm,-505.29mm)
.. controls ++(-45.45mm,-56.81mm) and ++(-2.03mm,48.78mm) .. ++(-50.51mm,-204.05mm)
;
\draw [fill=none,draw=black] (404.05mm,159.40mm)
.. controls ++(38.46mm,-38.64mm) and ++(-76.69mm,81.99mm) .. ++(125.23mm,-97.04mm)
;
\end{scope}
\end{tikzpicture}
\]

Let $M:=ord(T)$ for the ribbon fusion category $\C$. The rotation operator $\rho^{ \overline{a}^{l+m}\boxtimes \tilde{b}}_{n-(m+l), a}$ is diagonalizable with respect to some basis of eigenvectors $\{ b_i \}_{i= 1}^{M(n-(m+l)}$ ordered by the eigenvalue count established in the previous section. Therefore the operator $\theta_{a}\pi^{b}_{a}(A^{ n}_{l,m})$ is diagonalizable with basis $\{(V^{n}_{l, m})^{-1}(b_i)\}$. This implies:

\begin{cor}
\begin{enumerate}
    \item $\pi_{a}(A^{ n}_{l,m})$ is diagonalizable and $\pi_{a}(A^{ n}_{l,m})^{M(n-(m+l))}=1$.
    \item Given $\omega$ an $M(n-(m+l))^{th}$ root of unity, the complex number $\theta^{-1}_a \omega$ occurs as an eigenvalue of $\pi^{b}_{a,n} (A^{ n}_{l,m})$ with multiplicity $K^{\overline{a}^{l+m}\boxtimes \tilde{b}}_{n-(m+1), a}(\omega)$ 
\end{enumerate}
\end{cor}

Thus the computation of the eigenvalues for $A^{ n}_{l,m}$ reduces to computing the eigenvalues of appropriate rotation operators, which we know how to do in terms of Ng-Schauenburg indicators.  If $\ca$ happens to be modular, we can compute these eigenvalues in terms of the modular data $S$ and $T$, as we demonstrate explicitly in the next section. 

Now we suppose $\ca$ is modular with modular data $S,T$, with $ord(T)=M$. In this case, $G:\ca\boxtimes \tilde{\ca}\rightarrow Z(\ca)$ is an equivalence.  Again, assume $a$ is fixed so $K^{b}_{n,\omega}$ refers to the multiplicity of $\omega$ in $\rho^{b}_{n,a}$, where $\omega$ is a some $nM^{th}$ root of unity, and $b\in Z(\ca)$. Specializing the formulae from above to $n=2$, we see that for $\omega$ a $2M^{th}$ root of unity we have
\[
K^{c\boxtimes \tilde{b}}_{2,\omega}=\frac{\delta(\omega^{2},\frac{\theta_{b}}{\theta_{c}})}{2}\left( \omega^{-1}\nu^{c\boxtimes \tilde{b}}_{2,1}(a)+N^{b}_{\overline{c},a,a}\right).
\]
Recall from section 2 that $S_{c\boxtimes b, d\boxtimes e}= S_{c,d}S_{\overline{b},e}$, and $\theta_{a\boxtimes \tilde{b}}=\frac{\theta_a}{\theta_b}$, so we have the following formula, which was first derived in \cite[Proposition 6.1]{NS2}:
\[
\nu^{c\boxtimes \tilde{b}}_{2,1}(a)=\sum_{d, e\in \text{Irr}(\ca)} \left(\frac{\theta_{d}}{\theta_{e}} \right)^{2} S_{c,d}S_{\overline{b},e} N^{a}_{d,e} \quad .
\]
Combining these formulae gives
\begin{prop}\label{formula}  Let $\ca$ be a modular category with modular data $S, T$, and let $\omega$ be a $2M^{th}$ root of unity.   Then $\rho^{a\boxtimes  \tilde{b}}_{2,a}$ contains eigenvalue $\omega$ with multiplicity 
\[
K^{c\boxtimes  \tilde{b}}_{2,\omega}:=\delta_{\omega^{2}=\frac{\theta_{b}}{\theta_{c}}}\ \frac{1}{2}\left( \omega^{-1}\sum_{d, e\in \text{Irr}(\ca)} \left(\frac{\theta_{d}}{\theta_{e}} \right)^{2} S_{c,d}S_{\overline{b},e} N^{a}_{d,e}+N^{b}_{\overline{c},a,a}\right).
\]
In particular, $K^{c\boxtimes  \tilde{b}}_{2,\omega}$ is a non-negative integer, and $\omega\in \Spec(\rho^{c\boxtimes  \tilde{b}}_{2,a})$ if and only if $K^{c\boxtimes  \tilde{b}}_{2,\omega}>0$.
\end{prop}
\begin{rem*}
It was shown earlier by \cite[Theorem 8.3]{NS2} that the expression in parentheses in Proposition \ref{formula} is a non-negative integer.
\end{rem*}
\noindent It appears somewhat surprising that this number is a non-negative integer in general.  Using the Verlinde formula, we can compute the terms $N^{b}_{\overline{c},a,a}$ in the sum in terms of the $S$-matrix, and thus the above formula is entirely expressible in terms of the modular data.
In $B_{\infty}$, $\sigma_i$ is conjugate to $A_{0,0}^{2}$ and $\sigma_i \sigma_{i+1} \sigma_i$ is conjugate to $A_{1,0}^{3}$. This implies
\begin{enumerate}
\item $\Spec(\pi_{a}(\sigma_i))=\{\theta^{-1}_{a}\omega: K^{\mathbb{1}\boxtimes \tilde{ b}}_{2,\omega}\ne 0\ \text{for some}\ b\in \ca\}$.  
\item $\Spec(\pi_{a}(\sigma_{i}\sigma_{i+1}\sigma_{i}))=\{\theta^{-1}_{a}\omega : K^{\overline{a}\ \boxtimes\   \tilde{b}}_{2,\omega}\ne 0\ \text{for some}\ b\in \ca\}$.
\end{enumerate}

\begin{rem} Note that this allows us to compute the link invariants for the closures of powers of braids $A^{n}_{m,l}$ in terms of generalized FS-indicators, hence also in terms of the modular data when $\ca$ is modular.  For example, for the closure of the braid $(A^{ p}_{0,0})^{q}$ is the $(p,q)$ torus link.  And thus above observations given a formula for this invariant in terms of the dimensions, twists, and $\nu^{b}_{p,q}$.  This class of examples is directly computable from the GFS-indicators (without first determining the eigenvalues), but other examples will require more interesting interactions with the eigenvalues themselves.  It would be interesting to determine general classes of links whose invariants can be obtained this way, and precise formulae in terms of the modular data.
\end{rem}

\subsection{An example from the Haagerup fusion category}

As an example, we consider the quadratic fusion category $\mathcal{H}aag$ corresponding to the even part of the Haagerup subfactor.  This is the ``standard'' categorification of the fusion ring $H$ given in Section \ref{sec:fs-rigid}. Then the modular data for the center $Z(\mathcal{H}aag)$ is well known, and was first computed by Izumi in \cite{I2}.  We use here the form from \cite{EG}:

\tiny
$$S:=
\frac{1}{3}\left(
\begin{array}{cccccccccccc}
 x & 1-x & 1 & 1 & 1 & 1 & y & y & y & y & y & y \\
 1-x & x & 1 & 1 & 1 & 1 & -y & -y & -y & -y & -y & -y \\
 1 & 1 & 2 & -1 & -1 & -1 & 0 & 0 & 0 & 0 & 0 & 0 \\
 1 & 1 & -1 & 2 & -1 & -1 & 0 & 0 & 0 & 0 & 0 & 0 \\
 1 & 1 & -1 & -1 & -1 & 2 & 0 & 0 & 0 & 0 & 0 & 0 \\
 1 & 1 & -1 & -1 & 2 & -1 & 0 & 0 & 0 & 0 & 0 & 0 \\
 y & -y & 0 & 0 & 0 & 0 & c(1) & c(2) & c(3) & c(4) & c(5) & c(6) \\
 y & -y & 0 & 0 & 0 & 0 & c(2) & c(4) & c(6) & c(5) & c(3) & c(1) \\
 y & -y & 0 & 0 & 0 & 0 & c(3) & c(6) & c(4) & c(1) & c(2) & c(5) \\
 y & -y & 0 & 0 & 0 & 0 & c(4) & c(5) & c(1) & c(3) & c(6) & c(2)\\
 y & -y & 0 & 0 & 0 & 0 & c(5) & c(3) & c(2) & c(6) & c(1) & c(4) \\
 y & -y & 0 & 0 & 0 & 0 & c(6) & c(1) & c(5) & c(2) & c(4) & c(3)\\
\end{array}
\right)
$$

\bigskip

$$T:=
\left(
\begin{array}{cccccccccccc}
 1 & 0 & 0 & 0 & 0 & 0 & 0 & 0 & 0 & 0 & 0 & 0 \\
 0 & 1 & 0 & 0 & 0 & 0 & 0 & 0 & 0 & 0 & 0 & 0 \\
 0 & 0 & 1 & 0 & 0 & 0 & 0 & 0 & 0 & 0 & 0 & 0 \\
 0 & 0 & 0 & 1 & 0 & 0 & 0 & 0 & 0 & 0 & 0 & 0 \\
 0 & 0 & 0 & 0 & e^{\frac{2 i \pi }{3}} & 0 & 0 & 0 & 0 & 0 & 0 & 0 \\
 0 & 0 & 0 & 0 & 0 & e^{-\frac{2 i \pi }{3}} & 0 & 0 & 0 & 0 & 0 & 0 \\
 0 & 0 & 0 & 0 & 0 & 0 & e^{\frac{12 i \pi }{13}} & 0 & 0 & 0 & 0 & 0 \\
 0 & 0 & 0 & 0 & 0 & 0 & 0 & e^{-\frac{4 i \pi }{13}} & 0 & 0 & 0 & 0 \\
 0 & 0 & 0 & 0 & 0 & 0 & 0 & 0 & e^{\frac{4 i \pi }{13}} & 0 & 0 & 0 \\
 0 & 0 & 0 & 0 & 0 & 0 & 0 & 0 & 0 & e^{\frac{10 i \pi }{13}} & 0 & 0 \\
 0 & 0 & 0 & 0 & 0 & 0 & 0 & 0 & 0 & 0 & e^{-\frac{12 i \pi }{13}} & 0 \\
 0 & 0 & 0 & 0 & 0 & 0 & 0 & 0 & 0 & 0 & 0 & e^{-\frac{10 i \pi }{13}} \\
\end{array}
\right)
$$
\normalsize

\bigskip

\noindent where $x=\frac{13-3\sqrt{13}}{26},\ y=\frac{3}{\sqrt{13}}$, and $ c(j)=-2y\cos\left(\frac{2\pi j}{13}\right)$

The simple objects $\Irr(Z(\mathcal{H}aag))=\{x_i\; | \; i=1,\ldots, 12\}$ are indexed left to right along the columns of the $T$-matrix. As a demonstration of our formula, we consider the $6^{th}$ simple object $x_{6}$.  We chose this object because it has multiplicity 2 objects in the decomposition of its tensor square, but our formula is easy to compute for any object. See that:
\[
x_6 \otimes x_6 = x_1 \oplus 2x_2 \oplus x_3 \oplus x_4 \oplus x_5 \oplus 2x_6 \oplus x_7 \oplus x_8 \oplus x_9 \oplus x_{10} \oplus x_{11} \oplus x_{12}
\]
Then $\theta_{6}= e^{-\frac{2 i \pi }{3}}$, and the two entries in the ``possible eigenvalues column'' below are $\omega_{i}\theta^{-1}_{6}$, where $\omega_{i}=\pm \sqrt{\theta_{i}}$, and these are the possible eigenvalues of the braid acting on $\Hom_{\mathcal{H}aag}(x_{i}, x_{6}\otimes x_{6})$.  The final column indicates the multiplicity as computed by our formula.

$$
\begin{array}{|c|c|c|}
\hline
\textbf{object} & \textbf{possible eigenvalues} & \textbf{multiplicities}\\
\hline
1 & (-e^{\frac{2\pi i}{3}}, e^{\frac{2\pi i}{3}}) & (0,1)\\
\hline
2&(-e^{\frac{2\pi i}{3}}, e^{\frac{2\pi i}{3}}) & (1,1)\\
\hline
3&(-e^{\frac{2\pi i}{3}}, e^{\frac{2\pi i}{3}})&(0,1)\\
\hline
4&(-e^{\frac{2\pi i}{3}}, e^{\frac{2\pi i}{3}})&(0,1)\\
\hline
5&(1,-1)&(1,0)\\
\hline
6&(-e^{\frac{\pi i}{3}}, e^{\frac{\pi i}{3}}) &(2,0)\\
\hline
7&(e^{\frac{5\pi i}{39}}, -e^{\frac{5\pi i}{39}})&(1,0)\\
\hline
8&(-e^{\frac{20\pi i}{39}}, -e^{\frac{20\pi i}{39}})&(1,0)\\
\hline
9 & (-e^{\frac{32\pi i}{39}}, e^{\frac{32\pi i}{39}})&(1,0)\\
\hline
10&(e^{\frac{2\pi i}{39}}, -e^{\frac{2\pi i}{39}})&(0,1)\\
\hline
11&(-e^{\frac{8\pi i}{39}}, e^{\frac{8\pi i}{39}})&(1,0)\\
\hline
12&(-e^{\frac{11\pi i}{39}}, e^{\frac{11\pi i}{39}})&(0,1)\\
\hline
\end{array}
$$






\bibliographystyle{alpha}
\bibliography{bibliography}

\end{document}